\documentclass{article}
\usepackage[T1]{fontenc}
\usepackage{amssymb,amsmath,amsthm,bbm,enumerate,xr}
\usepackage{algorithm}
\usepackage{algorithmic}
\usepackage{stmaryrd}
\usepackage{dsfont}
\usepackage{version}
\usepackage{subfig}
\usepackage{ifthen}
\usepackage{graphicx,color}
\usepackage[numbers]{natbib}
\usepackage{algorithm,algorithmic}
\newcommand{\1}{\ensuremath{\mathbf{1}}}

\theoremstyle{plain} \newtheorem{theorem}{Theorem}[section]
\theoremstyle{plain} \newtheorem{proposition}[theorem]{Proposition}
\theoremstyle{plain} 
\theoremstyle{plain} \newtheorem{lemma}[theorem]{Lemma}
\theoremstyle{remark} \newtheorem{remark}[theorem]{Remark}

\newcounter{hypA}
\newenvironment{hypA}{\refstepcounter{hypA}\begin{itemize}
\item[{\bf A\arabic{hypA}}]}{\end{itemize}}

\newcommand{\osc}{\mathrm{osc}}

\newcommand{\eqdef}{\ensuremath{\stackrel{\mathrm{def}}{=}}}

\def\Xset{\mathbb{X}}
\def\Yset{\mathbb{Y}}
\def\Zset{\mathbb{Z}}

\newcommand{\pscal}[2]{\left\langle #1, #2 \right\rangle}
\newcommand{\K}{\mathcal{K}}

\newcommand{\limE}[2]{\mathsf{S}(#1,#2)}
\newcommand{\Rset}{\mathbb{R}}
\newcommand{\Nset}{\mathbb{N}}

\newcommand{\eqsp}{\;}
\newcommand{\jacG}{\Gamma}
\newcommand{\param}{\theta}

\newcommand{\paramset}{\Theta}
\newcommand{\rmd}{\mathrm{d}}

\newcommand{\rmL}{\mathrm{L}}
\newcommand{\bfy}{\mathbf{y}}
\newcommand{\bfY}{\mathbf{Y}}
\newcommand{\bfX}{\mathbf{X}}

\newcommand{\PP}{\mathbb{P}}
\newcommand{\Staset}{\mathcal{L}}

\newcommand{\loglik}[2]{\ell_{#2}^{#1}}

\newcommand{\CE}[4][]
{
\ifthenelse{\equal{#1}{}}{\mathbb{E}^{#2}_{#3}\left[#4\right]}{\mathbb{E}^{#2}_{#3}\left[#4\middle | #1\right]}
}

\newcommand{\lpnorm}[2]{\ensuremath{\left\| #1 \right\|_{#2}}}

\newcommand{\smoothfunc}[2]{\Phi_{#2}^{#1}}
\newcommand{\lyap}{\mathrm{W}}
\newcommand{\limT}[1]{\underset{#1\to +\infty}{\longrightarrow}}
\newcommand{\ps}[1]{#1\mathrm{-a.s.}}
\newcommand{\limEMmap}{\mathrm{G}}
\newcommand{\limEMmapt}{\mathrm{R}}
\newcommand{\mapS}{\bar{\mathrm{S}}}
\newcommand{\PPim}{\PP}
\newcommand{\Sset}{\mathcal{S}}

\newcommand{\shift}{\vartheta}

\newcommand{\averageparam}[1][]%
{\ifthenelse{\equal{#1}{}}{\ensuremath{\widetilde{\theta}}}{\ensuremath{\widetilde{\theta}}_{#1}}
}

\newcommand{\mix}[1][]%
{\ifthenelse{\equal{#1}{a}}{\alpha}{\beta}
}

\def\sigmaX{\mathcal{X}}
\def\sigmaY{\mathcal{Y}}

\newcommand{\adjfunc}[4][]
{\ifthenelse{\equal{#1}{}}{\ifthenelse{\equal{#4}{}}{\upsilon_{#2}}{\upsilon_{#2}(#4)}}
{\ifthenelse{\equal{#1}{smooth}}{\ifthenelse{\equal{#4}{}}{\tilde{\upsilon}_{#2}}{\tilde{\upsilon}_{#2}(#4)}}
{\ifthenelse{\equal{#1}{fully}}{\ifthenelse{\equal{#4}{}}{\upsilon^\star_{#2}}{\upsilon^\star_{#2}(#4)}}{\mathrm{erreur}}}}}
\newcommand{\kiss}[3][]
{\ifthenelse{\equal{#1}{}}{p_{#2}}
{\ifthenelse{\equal{#1}{fully}}{p^{\star}_{#2}}
{\ifthenelse{\equal{#1}{smooth}}{\tilde{r}_{#2}}{\mathrm{erreur}}}}}

\newcommand{\Kiss}[3][]
{\ifthenelse{\equal{#1}{}}{P_{#2}}
{\ifthenelse{\equal{#1}{fully}}{P^{\star}_{#2}}
{\ifthenelse{\equal{#1}{smooth}}{\tilde{R}_{#2}}{\mathrm{erreur}}}}}

\newcommand{\XinitIS}[2][]{\ifthenelse{\equal{#1}{}}{\ensuremath{\rho_{#2}}}{\ensuremath{\check{\rho}_{#2}}}}
\newcommand{\filt}[2][]%
{
\ifthenelse{\equal{#1}{}}{\ensuremath{\phi_{#2}}}%
{\ifthenelse{\equal{#1}{hat}}{\ensuremath{\phi^{N}_{#2}}}
{\ifthenelse{\equal{#1}{tilde}}{\ensuremath{\tilde{\phi}^{N}_{#2}}}
{\ifthenelse{\equal{#1}{tar}}{\ensuremath{\phi^{N,\mathrm{t}}_{#2}}}
{\ifthenelse{\equal{#1}{aux}}{\ensuremath{\phi^{N,\mathrm{a}}_{#2}}}
}
}
}
}
}
\newcommand{\M}{\ensuremath{M}}

\newcommand{\post}[3][]%
{
\ifthenelse{\equal{#1}{}}{\ensuremath{\phi_{#2|#3}}}%
{\ifthenelse{\equal{#1}{hat}}{\ensuremath{\phi^{N}_{#2|#3}}}
{\ifthenelse{\equal{#1}{tilde}}{\ensuremath{\tilde{\phi}^{N}_{#2|#3}}}
{\ifthenelse{\equal{#1}{tar}}{\ensuremath{\phi^{N,\mathrm{t}}_{#2|#3}}}}
}
}
}
\newcommand{\chunk}[4][]%
{\ifthenelse{\equal{#1}{}}{\ensuremath{{#2}_{#3:#4}}}{\ensuremath{#2^#1}_{#3:#4}}
}
\newcommand{\sumwght}[2][]{%
\ifthenelse{\equal{#1}{}}{\ensuremath{\Omega_{#2}}}{\ensuremath{\Omega_{#2}^{#1}}}}

\begin{document}
 
 \title{Online Expectation Maximization based algorithms for inference in Hidden Markov Models}
 \author{Sylvain Le Corff\,\footnote{LTCI, CNRS and TELECOM ParisTech, 46 rue Barrault 75634 Paris Cedex 13, France. sylvain.lecorff@telecom-paristech.fr}\;\,\footnote{This
    work is partially supported by the French National Research Agency, under
    the programs ANR-08-BLAN-0218 BigMC and ANR-07-ROBO-0002.}\, and Gersende Fort\,\footnote{LTCI, CNRS and TELECOM ParisTech, 46 rue Barrault 75634 Paris Cedex 13, France. gersende.fort@telecom-paristech.fr}}

\maketitle

\begin{abstract}
  The Expectation Maximization (EM) algorithm is a versatile tool for model
  parameter estimation in latent data models. When processing large data sets
  or data stream however, EM becomes intractable since it requires the whole
  data set to be available at each iteration of the algorithm.  In this
  contribution, a new generic online EM algorithm for model parameter inference
  in general Hidden Markov Model is proposed. This new algorithm updates the
  parameter estimate after a block of observations is processed (online).  The
  convergence of this new algorithm is established, and the rate of convergence
  is studied showing the impact of the block-size sequence. An averaging procedure is
  also proposed to improve the rate of convergence.  Finally, practical
  illustrations are presented to highlight the performance of these algorithms in comparison to other online maximum likelihood procedures.
\end{abstract}

\section{Introduction}
\label{BOEM:sec:intro}
A hidden Markov model (HMM) is a stochastic process $\{X_k,
Y_k\}_{k\geq 0}$ in $\Xset\times\Yset$, where the state sequence $\{X_k\}_{k\geq 0}$ is a Markov chain and where the observations $\{Y_k\}_{k\geq 0}$ are independent conditionally on $\{X_k\}_{k \geq 0}$. Moreover, the conditional distribution of $Y_k$ given the state sequence depends only on
$X_k$.  The sequence $\{X_k\}_{k\geq 0}$ being unobservable, any statistical inference task is carried out using the observations $\{Y_k\}_{k\geq 0}$. These HMM can be applied in a large
variety of disciplines such as financial econometrics
(\cite{mamon:elliott:2007}), biology
(\cite{churchill:1992}) or  speech recognition (\cite{juang:rabiner:1991}).

The Expectation Maximization (EM) algorithm is an iterative algorithm
used to solve maximum likelihood estimation in HMM, see~\cite{dempster:laird:rubin:1977}. The EM algorithm is generally simple to implement since it relies on complete data computations. Each iteration is decomposed into two steps: the E-step computes the conditional expectation of the complete data log-likelihood given the observations and the M-step updates the parameter estimate based on this conditional expectation. In many situations of interest, the complete data likelihood belongs to the curved exponential family. In this case, the E-step boils down to the computation of the conditional expectation of the complete data sufficient statistic. Even in this case, except for simple models such as linear Gaussian models or HMM with finite state-spaces, the E-step is intractable and has to be approximated e.g.  by Monte Carlo methods such as Markov Chain Monte Carlo methods or Sequential Monte Carlo
methods (see \cite{carlin:polson:stoffer:1992} or \cite{cappe:moulines:ryden:2005,doucet:defreitas:gordon:2001} and the references therein).

However, when processing large data sets or data streams, the EM algorithm might
become impractical.  {\em Online} variants of the EM algorithm have been first
proposed for independent and identically distributed (i.i.d.)  observations, see~\cite{cappe:moulines:2009}.  When the complete data likelihood belongs to the cruved exponential family, the E-step is replaced by a stochastic
approximation step while the M-step remains unchanged. The convergence of this online variant of the EM algorithm for
i.i.d.  observations is addressed by~\cite{cappe:moulines:2009}: the limit points are the stationary points of the Kullback-Leibler divergence between the marginal distribution of the
observation and the model distribution.

An online version of the EM algorithm for HMM when both the
observations and the states take a finite number of values (resp. when the
states take a finite number of values) was recently proposed by~\cite{mongillo:deneve:2008} (resp. by~\cite{cappe:2011}).  This algorithm has been extended to the case of general state-space models by
substituting deterministic approximation of the smoothing probabilities for
Sequential Monte Carlo algorithms (see~\cite{cappe:2009,delmoral:doucet:singh:2010a,lecorff:fort:moulines:2011}). There do not exist convergence results for these online EM algorithms for general state-space models (some insights on
the asymptotic behavior are nevertheless given in~\cite{cappe:2011}):
the introduction of many approximations at different steps of the algorithms
makes the analysis quite challenging.

In this contribution, a new online EM algorithm is proposed for HMM with complete data likelihood belonging to the curved exponential family. This algorithm sticks closely to the principles
of the original batch-mode EM algorithm. The M-step (and thus, the update of
the parameter) occurs at some deterministic times $\{T_k\}_{k \geq 1}$ i.e. we
propose to keep a fixed parameter estimate for blocks of observations of
increasing size.  More precisely, let $\{T_k\}_{k\geq0}$ be an increasing
sequence of integers $(T_0=0)$. For each $k\geq0$, the parameter's value is
kept fixed while accumulating the information brought by the observations
$\{Y_{T_{k}+1},\cdots,Y_{T_{k+1}}\}$.  Then, the parameter is updated at the end of the block.
This algorithm is an online algorithm since the sufficient statistics of the
$k$-th block can be computed on the fly by updating an intermediate quantity
when a new observation $Y_t$, $t \in \{T_{k}+1, \dots, T_{k+1}\}$ becomes
available. Such recursions are provided in recent works
  on online estimation in HMM, see \cite{cappe:2009,cappe:2011,delmoral:doucet:singh:2010a}.

  This new algorithm, called {\em Block Online EM} (BOEM) is derived
  in Section~\ref{BOEM:sec:BOEM:description} together with an {\em averaged}
  version.  Section~\ref{BOEM:sec:MCexperiments} is devoted to practical
  applications: the BOEM algorithm is used to perform parameter inference in HMM where the
  forward recursions mentioned above are available explicitly. In the case of finite
  state-space HMM, the BOEM algorithm is compared to a gradient-type recursive maximum
  likelihood procedure and to the online EM algorithm of~\cite{cappe:2011}. The convergence of the BOEM algorithm is addressed in Section~\ref{BOEM:sec:convergence}.  The BOEM algorithm is seen as a perturbation of a deterministic {\em limiting EM} algorithm which is shown to converge to the stationary points of the limiting relative entropy (to which the true parameter belongs if the model is well specified). The perturbation is shown to vanish (in some sense) as the number of
  observations increases thus implying that the BOEM algorithms inherits the asymptotic
  behavior of the limiting EM algorithm.  Finally, in Section~\ref{BOEM:sec:averaging}, we
  study the rate of convergence of the BOEM algorithm as a function of the block-size sequence. We prove that the averaged BOEM algorithm is rate-optimal when the block-size sequence grows polynomially.  All the proofs are
  postponed to Section~\ref{BOEM:sec:proofs}; supplementary proofs and comments are provided
  in~\cite{lecorff:fort:2011-supp}.

\section{The Block Online EM algorithms}
\label{BOEM:sec:BOEM:description}
\subsection{Notations and Model assumptions}
Our model is defined as follows. Let $\paramset$ be a compact subset of $\Rset^{d_\theta}$. We are given a family of transition kernels
$\{M_{\param}\}_{\param\in\paramset}$, $M_{\param}:\Xset\times\sigmaX\to[0,1]$, a positive $\sigma$-finite
measure $\mu$ on $(\Yset,\sigmaY)$, and a family of transition densities with respect to $\mu$, $\{g_{\param}\}_{\param\in\paramset}$,  $g_{\param}:\Xset\times\Yset\to\mathbb{R}_+$.  For each
$\param\in\paramset$, define the transition kernel $K_\theta$ on $\Xset\times\Yset$ by
\[
	K_\theta \left[(x,y), C\right] \eqdef
	\int \1_C(x',y') \, g_\theta(x',y')\,\mu(\rmd y')\,
	\M_\theta(x,\rmd x')\eqsp.
\]
Denote by $\{X_k,Y_k\}_{k\ge 0}$ the canonical
coordinate process on the measurable space $\left((\Xset\times\Yset)^\Nset,(\sigmaX\otimes\sigmaY)^{\otimes\Nset}\right)$. For any $\param\in\paramset$ and any probability distribution $\chi$ on $(\Xset,\sigmaX)$, let $\PP^\chi_\theta$ be the probability distribution on $((\Xset\times\Yset)^\Nset,(\sigmaX\otimes\sigmaY)^{\otimes\Nset})$ such that $\{X_k,Y_k\}_{k\ge 0}$ is Markov chain with initial distribution $\PP^\chi_\param((X_0,Y_0)\in C)=\int
\1_C(x,y)\,g_\param(x,y)\, \mu(\rmd y)\,\chi(\rmd x)$ and transition
kernel $K_\theta$. The expectation with
respect to $\PP^\chi_\param$ is denoted by $\mathbb{E}^\chi_\theta$. Throughout this paper, it is assumed that the Markov transition kernel $K_\param$ has a unique invariant distribution $\pi_\param$ (see below for further comments). For the stationary Markov chain with initial distribution $\pi_\param$, we write $\PP_\param$ and $\mathbb{E}_\theta$ instead of $\PP^{\pi_\param}_\param$ and $\mathbb{E}^{\pi_\param}_\theta$. Note also that the stationary Markov chain $\{X_k,Y_k\}_{k\ge 0}$ can be extended to a two-sided Markov chain $\{X_k,Y_k\}_{k\in \Zset}$.

It is assumed that, for any $\param\in\paramset$ and any $x\in\Xset$, $M_{\param}(x,\cdot)$ has a density $m_{\param}(x,\cdot)$ with respect to a finite measure $\lambda$ on $(\Xset,\sigmaX)$. Define the complete data likelihood by
\begin{equation}
\label{BOEM:eq:completelik}
p_{\param}(x_{0:T},y_{0:T}) \eqdef g_{\param}(x_{0},y_{0})\prod_{i=0}^{T-1}m_\param(x_i,
  x_{i+1})g_{\param}(x_{i+1},y_{i+1})\eqsp,
\end{equation}
where, for any $u \leq s$, we will use the shorthand notation $x_{u:s}$ for the sequence $(x_u,
\cdots, x_s)$. For any probability distribution $\chi$ on $(\Xset,\sigmaX)$, any $\param\in\paramset$ and any $s\le u\le v\le t$, we have
\[
\CE[Y_{s:t}]{\chi}{\param}{f(X_{u:v})} = \int f(x_{u:v})\filt{\param,u:v|s:t}^{\chi}(\rmd x_{u:v})\eqsp,
\]
where $\filt{\param,u:v|s:t}^{\chi}$ is the so-called fixed-interval smoothing distribution. We also define the  fixed-interval smoothing distribution when $X_s\sim \chi$:
\begin{multline}
\label{eq:deflaws:shift}
\CE[Y_{s+1:t}]{\chi,s}{\param}{f(X_{u:v})} \\
=\frac{\int \prod_{i=s+1}^{t}\{m_{\param}(x_{i-1},x_{i})g_{\param}(x_{i},Y_{i})\}f(x_{u:v})\chi(\rmd x_{s})\lambda(\rmd x_{s+1:t})}{\int \prod_{i=s+1}^{t}\{m_{\param}(x_{i-1},x_{i})g_{\param}(x_{i},Y_{i})\}\chi(\rmd x_{s})\lambda(\rmd x_{s+1:t})}\eqsp.
\end{multline}

Given an initial distribution $\chi$ on $(\Xset,\sigmaX)$ and $T+1$ observations $Y_{0:T}$, the EM algorithm maximizes the so-called incomplete data log-likelihood $\param\mapsto \ell_{\param,T}^{\chi}$ defined by
\begin{equation}
\label{BOEM:eq:loglik}
\ell_{\param,T}^{\chi}(\bfY) \eqdef \log\int p_{\param}(x_{0:T},Y_{1:T})\chi(\rmd x_{0})\lambda(\rmd x_{1:T})\eqsp.
\end{equation}
The central concept of the EM algorithm is that the intermediate quantity defined by
\[
\param\mapsto Q(\param,\param')\eqdef \CE[Y_{1:T}]{\chi}{\param'}{\log p_{\param}(X_{0:T},Y_{1:T})}
\]
may be used as a surrogate for $\ell_{\param,T}^{\chi}(Y_{0:T})$ in the maximization procedure. Therefore, the EM algorithm iteratively builds a sequence $\{\param_{n}\}_{n\ge 0}$ of parameter estimates following the two steps:
\begin{enumerate}[i)]
\item \label{BOEM:Estep}Compute $\param\mapsto Q(\param,\param_{n})$.
\item \label{BOEM:Mstep}Choose $\param_{n+1}$ as a maximizer of $\param \mapsto Q(\param,\param_{n})$.
\end{enumerate}
In the sequel,  it is assumed that  there exist functions $S$, $\phi$ and $\psi$ such that (see A\ref{BOEM:assum:exp} for a more precise definition), for any $(x,x')\in\Xset^{2}$ and any $y\in\Yset$,
\[
m_\param(x,x') g_\param(x',y) = \exp\left\{\phi(\param) +
  \pscal{S(x,x,',y)}{\psi(\param)}\right\}\eqsp.
  \]
 Therefore, the complete data likelihood belongs to the curved exponential family and the step \ref{BOEM:Estep}) of the EM algorithm amounts to computing
  \[
  \param\mapsto Q(\param,\param_{n}) = \phi(\param) + \pscal{\frac{1}{T}\sum_{t=1}^{T}\CE[Y_{1:T}]{\chi}{\param_{n}}{S(X_{t-1},X_{t},Y_{t})}}{\psi(\param)}\eqsp,
  \]
  where $\pscal{\cdot}{\cdot}$ is the scalar product on $\Rset^{d}$ (and where the contribution of $g_{\param}(x_0,Y_0)$ is omitted for brevity).
  It is also assumed that for any $s\in\Sset$, where $\Sset$ is an appropriately defined set, the function $\param\mapsto \phi(\param) + \pscal{s}{\psi(\param)}$ has a unique maximum denoted by $\bar\param(s)$. Hence, a step of the EM algorithm writes
  \[
  \param_n = \bar \param\left(\frac{1}{T}\sum_{t=1}^{T}\CE[Y_{1:T}]{\chi}{\param_{n-1}}{S(X_{t-1},X_{t},Y_{t})}\right)\eqsp.
  \]
\subsection{The Block Online EM (BOEM) algorithms}
We now derive an online version of the EM algorithm. Define $\bar S_{\tau}^{\chi,T}(\param, \bfY)$ as the intermediate quantity of the EM algorithm computed with the observations $Y_{T:T+\tau}$:
\begin{equation}
  \label{BOEM:eq:rewrite:barS}
  \bar S_{\tau}^{\chi,T}(\param, \bfY) \eqdef
\frac{1}{\tau} \sum_{t=T+1}^{T+\tau}\CE[Y_{T+1:T+\tau}]{\chi,T}{\param}{S(X_{t-1},X_{t},Y_{t})}\eqsp,
\end{equation}
where $\CE[Y_{T+1:T+\tau}]{\chi,T}{\param}{\cdot}$ is defined by \eqref{eq:deflaws:shift}.
Let $\{\tau_n \}_{n \geq 1}$ be a sequence of positive integers such that $\lim_{n\to \infty}\tau_{n}=+\infty$ and set
\begin{equation}
\label{BOEM:eq:timeupdate}
T_n\eqdef \sum_{k=1}^n \tau_k\quad\mbox{and}\quad T_0 \eqdef 0\eqsp;
\end{equation}
$\tau_n$ denotes the length of the $n$-th block. Given an initial value $\param_{0} \in \paramset$, the BOEM algorithm defines a sequence $\{\param_{n}\}_{n \geq 1}$ by
\begin{equation}
\label{BOEM:eq:Bonem:recursion}
\param_{n}  \eqdef \bar\param \left[S_{n-1}\right]\eqsp,\;\mbox{and}\; S_{n-1}\eqdef \bar S_{\tau_n}^{\chi_{n-1},T_{n-1}}(\param_{n-1}, \bfY)\eqsp,
\end{equation}
where $\{\chi_n\}_{n\ge 0}$ is a family of probability distributions on $(\Xset,\sigmaX)$. By analogy to the regression problem, an estimator with reduced variance can be obtained by averaging and weighting the
successive estimates (see
\cite{kushner:yin:1997,polyak:juditsky:1992} for a discussion on
the averaging procedures). Define $\Sigma_0 \eqdef 0$ and for $n \geq 1$,
\begin{equation}
\label{BOEM:eq:Bonem:averaged}
\Sigma_{n} \eqdef
\frac{1}{T_n} \sum_{j=1}^n \tau_j  \, S_{j-1}\eqsp.
\end{equation}
Note that this quantity can be computed iteratively and does not require to
store the past statistics $\{S_{j}\}_{j=0}^{n-1}$.  Given an initial value $\widetilde
\param_{0}$, the averaged BOEM algorithm defines a sequence $\{\widetilde
\param_{n}\}_{n \geq 1}$ by
\begin{equation}
  \label{BOEM:eq:abonem:recursion}
  \widetilde \param_n \eqdef \bar\param \left(\Sigma_n\right) \eqsp.
\end{equation}
The algorithm above relies on the assumption that $S_{n}$ can be computed in closed form. In the HMM case, this property is satisfied only for linear Gaussian models or when the state-space is finite. In all other cases, $S_{n}$ cannot be computed
explicitly and will be replaced by a Monte Carlo approximation $\widetilde{S}_{n}$. Several Monte Carlo approximations can be used to compute $\widetilde{S}_{n}$. The convergence properties of the Monte Carlo BOEM algorithms rely on the assumption that the Monte Carlo error can be controlled on each block. \cite{lecorff:fort:2012} provides examples of applications when Sequential Monte Carlo algorithms are used. Hereafter, we use the same notation $\{\param_{n}\}_{n\ge 0}$ and $\{\widetilde\param_{n}\}_{n\ge 0}$ for the original BOEM algorithm or its Monte Carlo approximation.

Our algorithms update the parameter after processing a block of observations.  Nevertheless, the intermediate quantity $S_{n}$ can be either exactly computed or approximated in such a way that the observations are processed online. In this case, the intermediate quantity $S_n$ or $\widetilde S_n$ is updated online for each observation. Such an algorithm is described in \cite[Section~$2.2$]{cappe:2011} and \cite[Proposition~$2.1$]{delmoral:doucet:singh:2010a} and can be applied either to finite state-space HMM or to linear Gaussian models. \cite{delmoral:doucet:singh:2010a} proposed a Sequential Monte Carlo approximation to compute $\widetilde{S}_{n}$ online for more complex models (see also \cite{lecorff:fort:2012}).

The classical theory of maximum likelihood estimation often relies on the assumption that the "true" distribution of the observations belongs to the specified parametric family of
distributions. In many cases, it is doubtful that this assumption is satisfied.
It is therefore natural to investigate the convergence of the BOEM algorithms and to identify the possible limit
for misspecified models i.e. when the observations $\{Y_k\}_{k\ge 0}$ are from an ergodic process which is not necessarily an HMM.

\section{Application to inverse problems in Hidden Markov Models}
\label{BOEM:sec:MCexperiments}
In Section~\ref{BOEM:sec:LGM}, the performance of the BOEM algorithm and its averaged version are illustrated in a linear Gaussian model. In Section~\ref{BOEM:sec:appli:finiteHMM}, the BOEM algorithm is compared to online maximum likelihood procedures in the case of finite state-space HMM.

Applications of the Monte Carlo BOEM algorithm to more complex models with Sequential Monte Carlo methods can be found in \cite{lecorff:fort:2012}.

\subsection{Linear Gaussian Model}
\label{BOEM:sec:LGM}
Consider the linear Gaussian model:
\begin{equation*}
  X_{t+1} = \phi X_t + \sigma_uU_t\eqsp, \qquad \qquad Y_t = X_t + \sigma_vV_t\eqsp,
\end{equation*}
where $X_0\sim\mathcal{N}\left(0,\sigma_u^2 (1-\phi^2)^{-1}\right)$,
$\{U_t\}_{t\geq 0},\{V_t\}_{t\geq 0}$ are independent i.i.d. standard Gaussian r.v., independent from
$X_0$. Data are sampled using $\phi = 0.9$, $\sigma_{u}^{2} = 0.6$ and
$\sigma_{v}^{2} = 1$. All runs are started with $\phi = 0.1$, $\sigma_{u}^{2} =
1$ and $\sigma_{v}^{2} = 2$.

We illustrate the convergence of the BOEM algorithms. We choose $\tau_n = n^{1.1}$. We display in Figure~\ref{BOEM:fig:LGM} the median and lower and upper quartiles  for the estimation of $\phi$ obtained with $100$ independent Monte Carlo
experiments. Both the BOEM algorithm and its averaged version converge to the true value $\phi=0.9$; the averaging procedure clearly improves the variance of the estimation.

\begin{figure}[!h]
  \centering
  \subfloat[The BOEM algorithm without averaging.]{\label{BOEM:fig:LGMnonaveraging}\includegraphics[width=.75\textwidth]{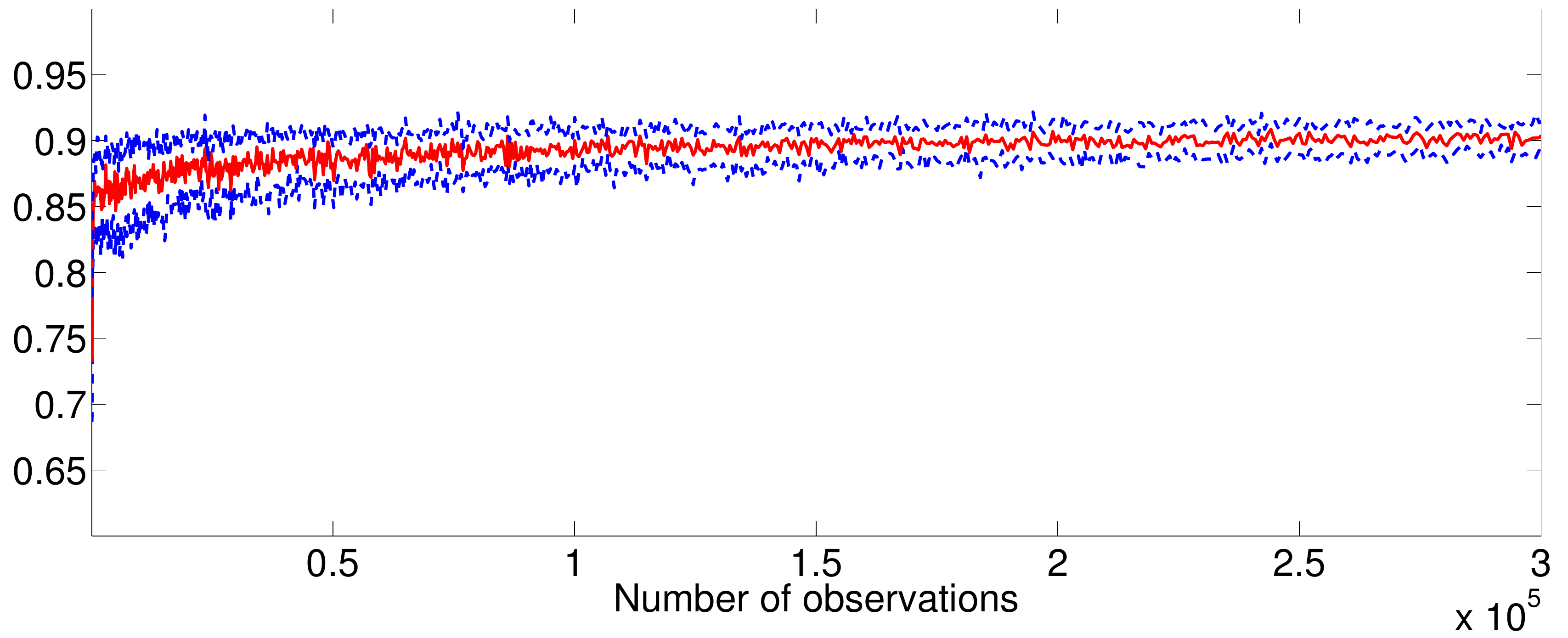}}\\
  \subfloat[The BOEM algorithm with averaging.]{\label{BOEM:fig:LGMaveraging}\includegraphics[width=.75\textwidth]{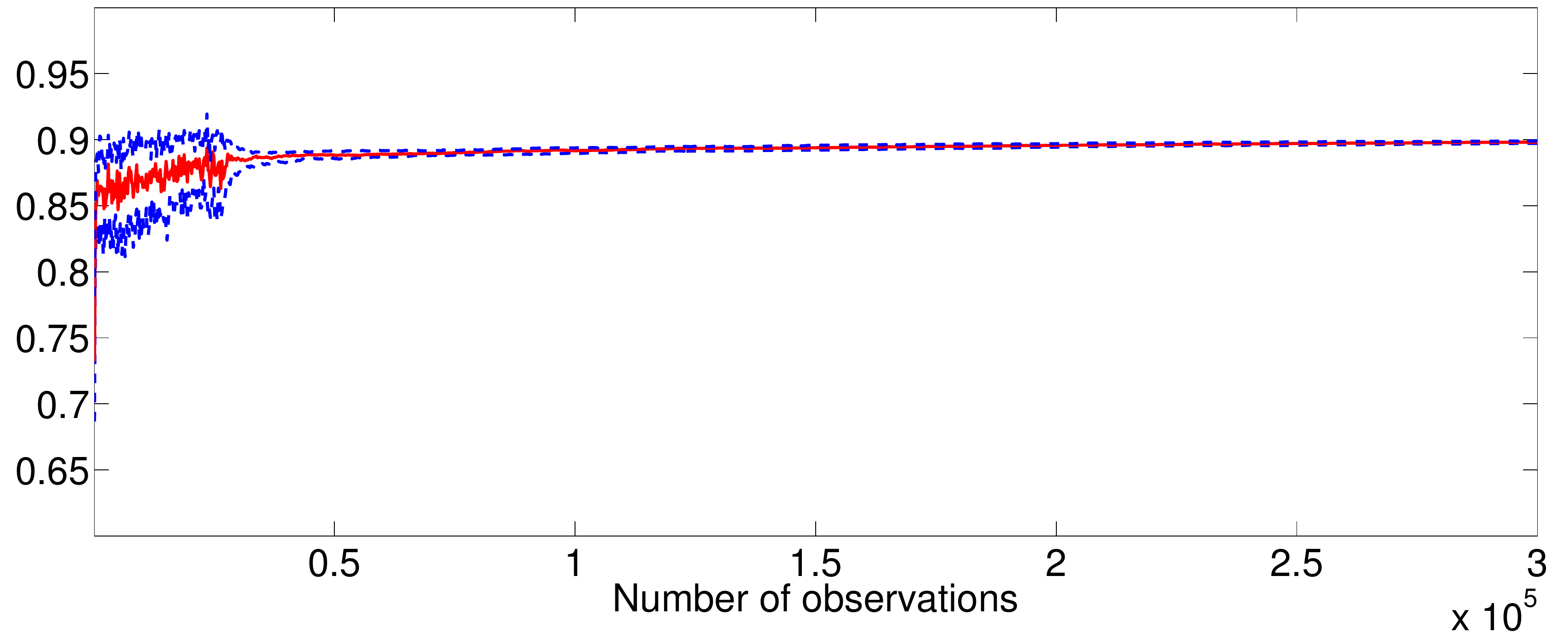}}
   \caption{Estimation of $\phi$.}
   \label{BOEM:fig:LGM}
\end{figure}

We now discuss the role of $\{\tau_{n}\}_{n\geq 0}$.
Figure~\ref{BOEM:fig:LGMblocksize} displays the empirical variance, when estimating
$\phi$, computed with $100$ independent Monte Carlo runs, for different numbers
of observations and, for both the BOEM algorithm and its averaged version. We consider
four polynomial rates $\tau_n \sim n^{b}$, $b \in \{1.2, 1.8, 2, 2.5 \}$.
Figure~\ref{BOEM:fig:LGMblocksize:nonavg} shows that the choice of
$\{\tau_{n}\}_{n\geq 0}$ has a great impact on the empirical variance of the
(non averaged) BOEM path $\{\theta_n\}_{n \geq 0}$.  To reduce this
variability, a solution could consist in increasing the block sizes $\tau_n$ at
a larger. The influence of the block size sequence $\tau_n$ is greatly reduced with the averaging procedure as shown in Figure~\ref{BOEM:fig:LGMblocksize:avg}. We will show in
Section~\ref{BOEM:sec:averaging} that averaging really improves the rate of
convergence of the BOEM algorithm.
 \begin{figure}[!h]
   \centering
   \subfloat[The BOEM algorithm, without averaging]{\label{BOEM:fig:LGMblocksize:nonavg}\includegraphics[width=0.7\textwidth]{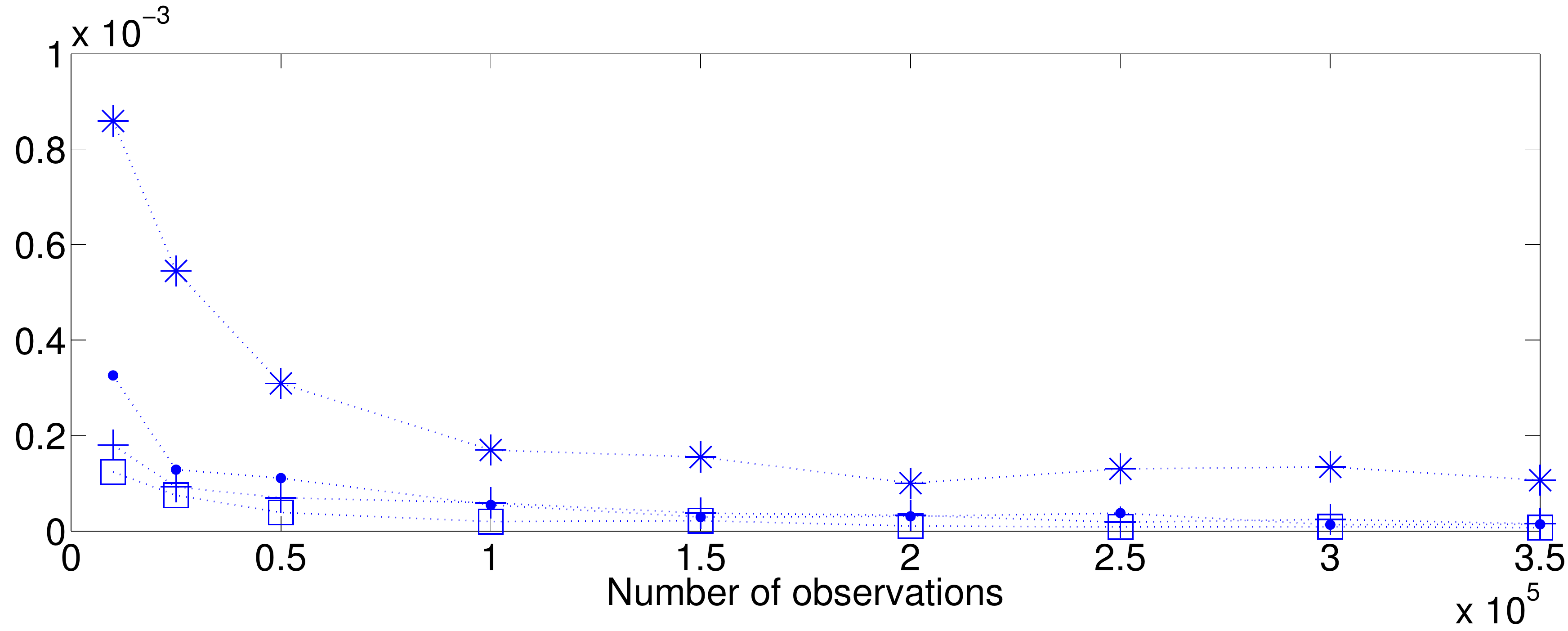}}\\
   \subfloat[The BOEM algorithm, with averaging]{\label{BOEM:fig:LGMblocksize:avg}\includegraphics[width=0.7\textwidth]{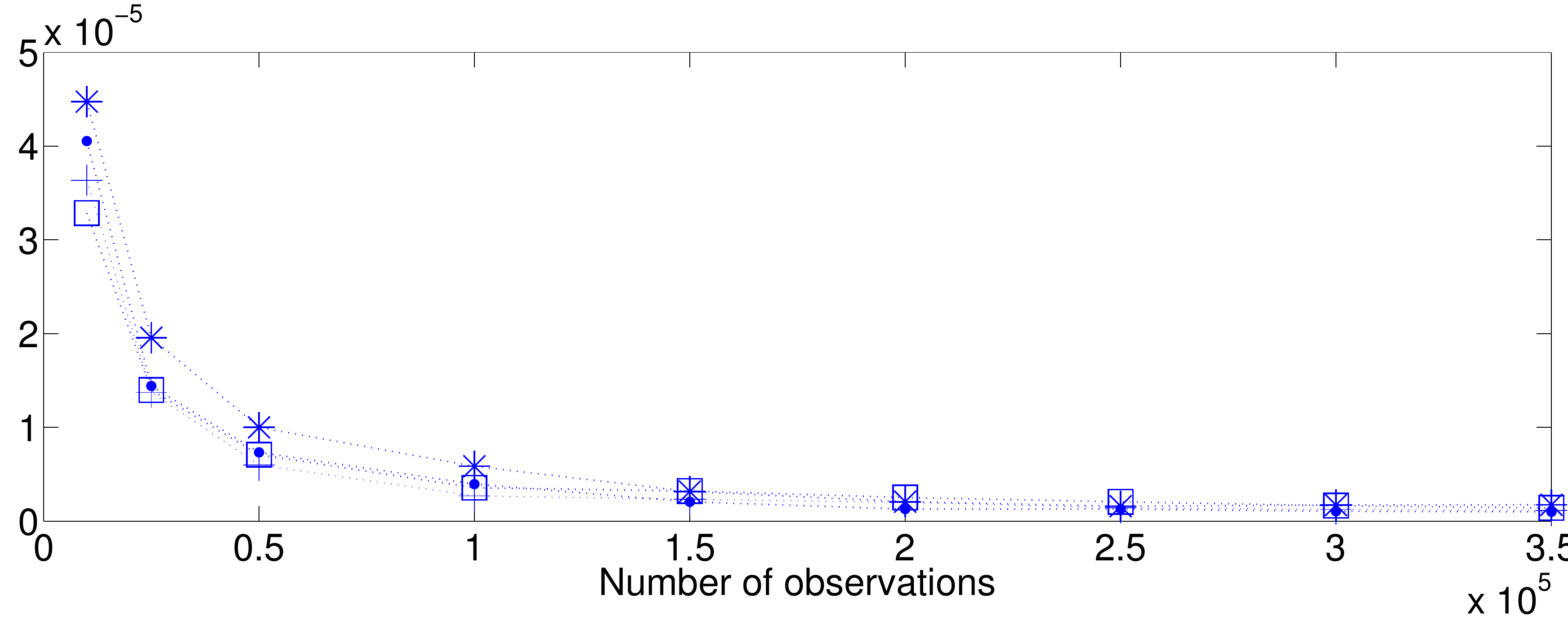}}
   \caption{The BOEM algorithm: empirical variance of the estimation of $\phi$ after $n = 0.5 \ell \, \cdot10^5$  observations ($\ell \in \{1,\cdots,7\}$) for different block size schemes $\tau_{n} \sim n^{1.2}$ (stars), $\tau_{n} \sim n^{1.8}$ (dots), $\tau_{n} \sim n^{2}$ (crosses) and $\tau_{n} \sim n^{2.5}$ (squares).}
   \label{BOEM:fig:LGMblocksize}
 \end{figure}

\subsection{Finite state-space HMM}
\label{BOEM:sec:appli:finiteHMM}
  We consider a Gaussian mixture process with Markov dependence of the form: $Y_t = X_t + V_t$ where $\{X_{t}\}_{t\geq 0}$ is a Markov chain taking values in
$\Xset\eqdef\{x_{1}, \dots, x_{d}\}$, with initial distribution $\chi$ and a $d \times
d$ transition matrix $m$. $\{V_t\}_{t\geq 0}$ are i.i.d.  $\mathcal{N}(0,v)$ r.v., independent from $\{X_{t}\}_{t\geq 0}$, i.e., for all $(x,y)\in\Xset\times \Yset$,
\[
g_{\param}(x,y) \eqdef (2\pi v)^{-1/2}\exp\left\{-\frac{(y-x)^{2}}{2v}\right\}\eqsp,
\]
where $\param\eqdef \left(v, x_{1:d}, (m_{i,j})_{i,j=1}^{d}\right)$. The true transition matrix is given by
\[
m =  \begin{pmatrix}0.5 & 0.05 & 0.1 & 0.15 & 0.15 & 0.05\\ 
    0.2& 0.35  &0.1 &0.15 &0.05& 0.15\\
          0.1& 0.1  &0.6 & 0.05 &0.05 &0.1\\
          0.02 &0.03& 0.1 &0.7 &0.1& 0.05\\
          0.1 &0.05& 0.13 &0.02 &0.6 &0.1\\
          0.1& 0.1& 0.13 &0.12& 0.1& 0.45\end{pmatrix}\eqsp.
\]
In the experiments below, the initial distribution below is chosen as the uniform distribution on $\Xset$. The statistics used to estimate $\param$ are, for all $(i,j)\in\{1,\cdots,d\}$ and all $(x,x')\in\Xset^{2}$,
\begin{align}
S^{i,0}(x,x',y) &= \1_{x_{i}}(x')\eqsp,\quad\hspace{.4cm} S^{i,1}(x,x',y) = y\1_{x_{i}}(x')\eqsp,\label{BOEM:eq:stats}\\
S^{i,2}(x,x',y) &= y^{2}\1_{x_{i}}(x')\eqsp,\quad S_{i,j}(x,x',y) = \1_{x_{i}}(x)\1_{x_{j}}(x')\nonumber\eqsp.
\end{align}
The online computation of these intermediate quantities is given \cite[Section~$2.2$]{cappe:2011}. The computations below are performed for each statistic in \eqref{BOEM:eq:stats}. Define, for all $x\in\Xset$, $\phi_{0}(x)=\chi(x)$ and $\rho_{0}(x)=0$.
\begin{enumerate}[i)]
\item For $t\in\{1,\cdots,\tau\}$, compute, for any $x\in\Xset$,
\[
\phi_{t}(x) = \frac{\sum_{x'\in \Xset}\phi_{t-1}(x')m_{x',x}g_{\param}(x,Y_{t+T})}{\sum_{x',x''\in \Xset}\phi_{t-1}(x')m_{x',x''}g_{\param}(x'',Y_{t+T})}\eqsp,
\]
and
\[
r_{t}(x,x') = \frac{\phi_{t-1}(x')m_{x',x}}{\sum_{x''\in\Xset}\phi_{t-1}(x'')m_{x'',x}}\eqsp.
\]
\begin{equation*}
\rho_{t}(x) = \sum_{x'\in\Xset}\left[\frac{1}{t}S(x,x',Y_{t+T}) + \left(1-\frac{1}{t}\right)\rho_{t-1}(x')\right]r_{t}(x,x')\eqsp.
\end{equation*}
\item Set
\[
\bar S_{\tau}^{\chi,T}(\param,\bfY) = \sum_{x\in\Xset}\rho_{\tau}(x)\phi_{\tau}(x)\eqsp.
\]
\end{enumerate}
At the end of the block, the new estimate is given, for all $(i,j)\in\{1,\cdots,d\}^{2}$ by (the dependence on $\bfY$, $\param$, $\chi$, $T$ and $\tau$ is dropped from the notation)
\[
m_{i,j} = \frac{\bar S_{i,j}}{\sum_{j=1}^{d}\bar S_{i,j}}\eqsp,\; x_{i} = \frac{\bar S^{i,1}}{\bar S^{i,0}},\; v =\sum_{i=1}^{d}\bar S^{i,2}+\sum_{i=1}^{d}x_{i}^{2}\bar S^{i,0}-2\sum_{i=1}^{d}x_{i}\bar S^{i,1}\eqsp.
\]

We first compare the averaged BOEM algorithm to the online EM (OEM) procedure of \cite{cappe:2011} combined with a Polyak-Ruppert averaging (see \cite{polyak:juditsky:1992}). Note that the convergence of the OEM algorithm is still an open problem. In this case, we want to estimate the variance $v$ and the states $\{x_{1}, \dots, x_{d}\}$. All the runs are started from $v = 2$ and from the initial states $\{-1;0;.5;2;3;4\}$.  The algorithm in \cite{cappe:2011} follows a stochastic approximation update and depends on a step-size sequence $\{\gamma_{n}\}_{n\geq 0}$. It is expected that the rate of convergence in $\rmL_{2}$ after $n$ observations is $\gamma_n^{1/2}$ (and $n^{-1/2}$ for its averaged version) - this assertion relies on classical results for stochastic approximation.  We prove in Section~\ref{BOEM:sec:averaging} that the rate of convergence of the BOEM algorithm is
$n^{-b/(2(b+1))}$ (and $n^{-1/2}$ for its averaged
version) when $\tau_n \propto  n^b$. Therefore, we set $\tau_{n} = n^{1.1}$ and $\gamma_{n} = n^{-0.53}$. Figure~\ref{BOEM:fig:quantilev} displays the empirical median and first and last quartiles for the estimation of $v$ with both algorithms and their averaged versions as a function of the number of observations. These estimates are obtained over $100$ independent Monte Carlo runs. Both the BOEM and the OEM algorithms converge to the true value of $v$ and the averaged versions reduce the variability of the estimation.  Figure~\ref{BOEM:fig:quantilex1} shows the similar behavior of both averaged algorithms for the estimation of $x_{1}$ in the same experiment. Some supplementary graphs on the estimation of the states can be found in \cite[Section~$4$]{lecorff:fort:2011-supp}).

 \begin{figure}[!h]
   \centering
   \subfloat[The BOEM algorithm.]{\includegraphics[width=0.5\textwidth]{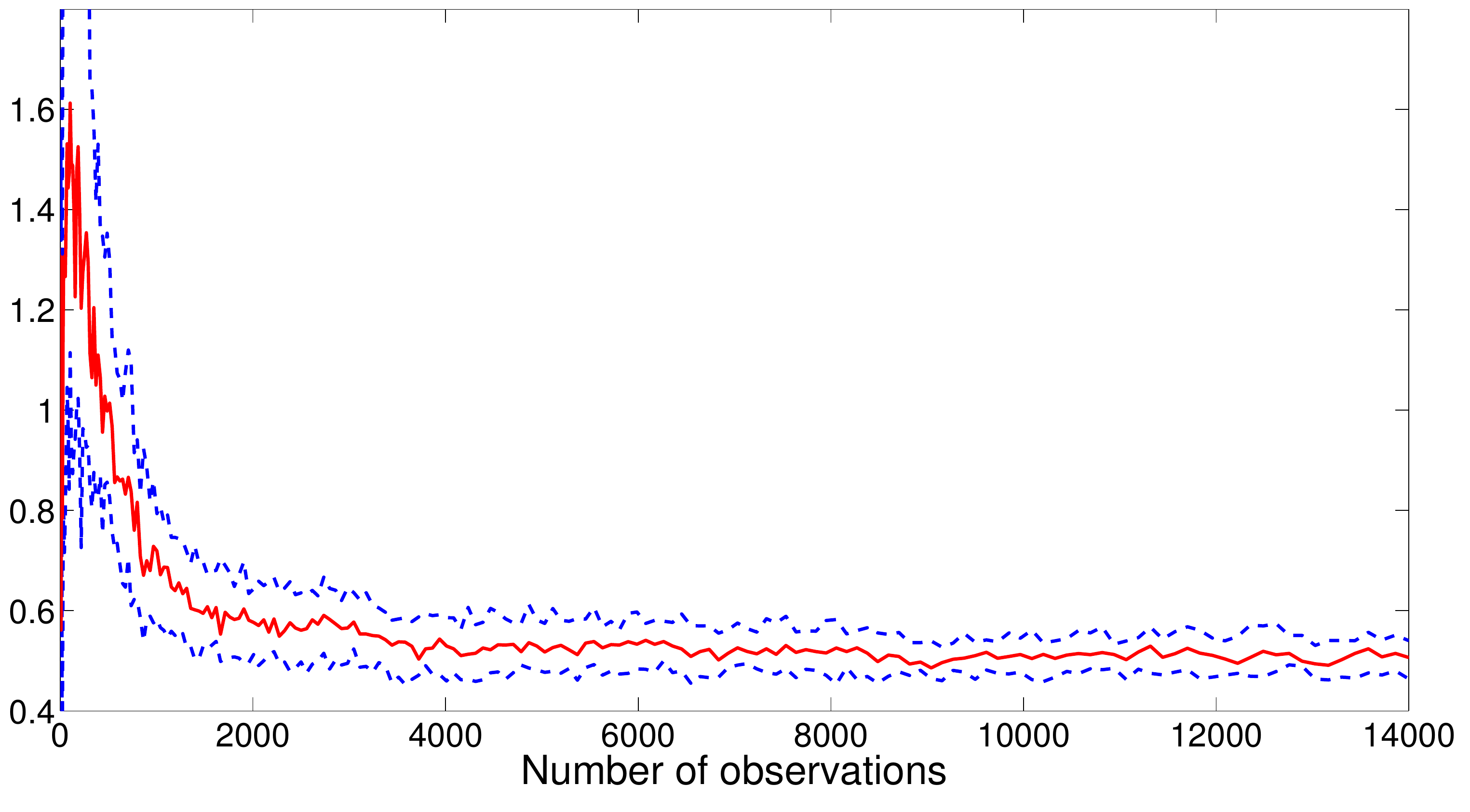}}
   \subfloat[The OEM algorithm.]{\includegraphics[width=0.5\textwidth]{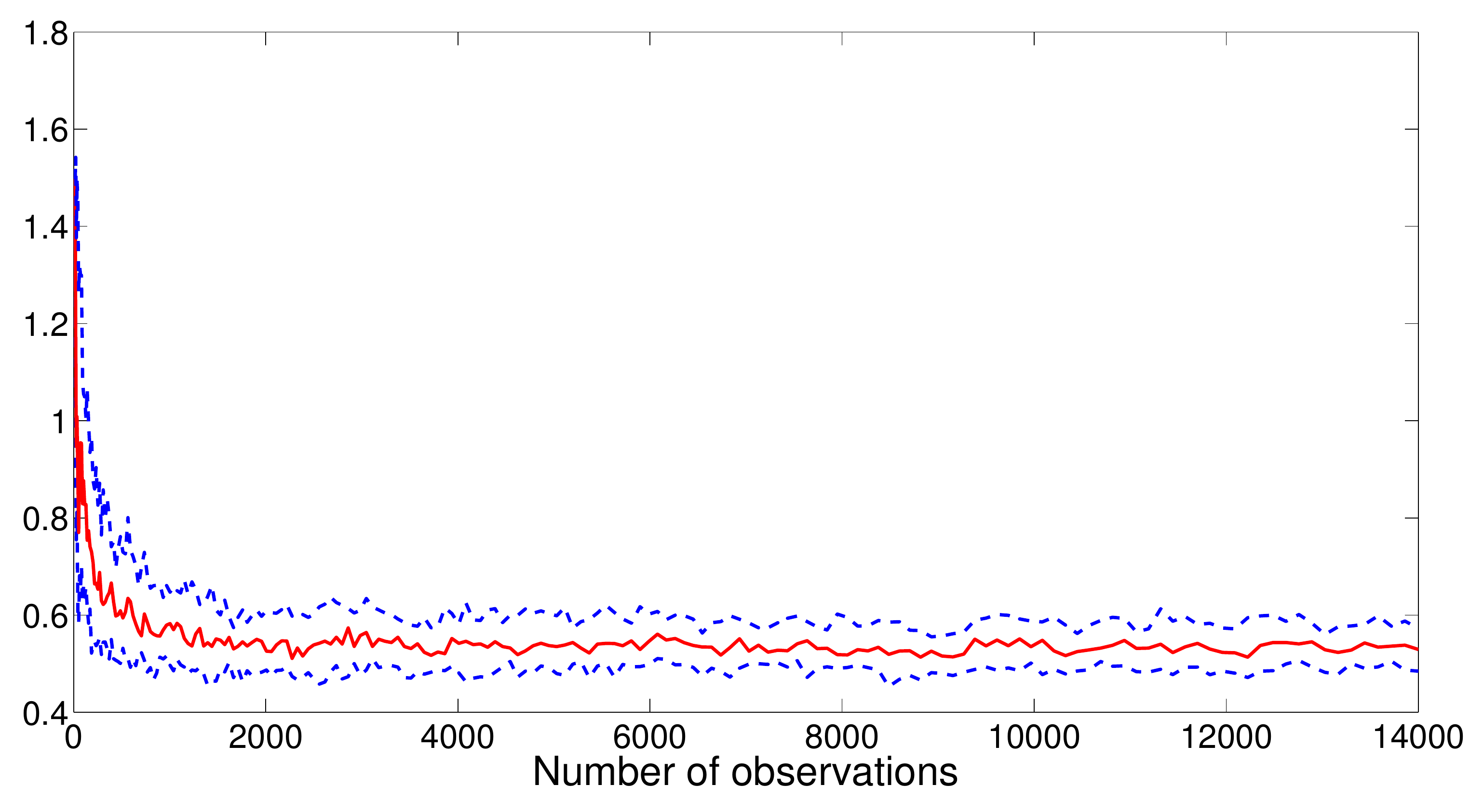}}\\
   \subfloat[The averaged BOEM algorithm.]{\includegraphics[width=0.5\textwidth]{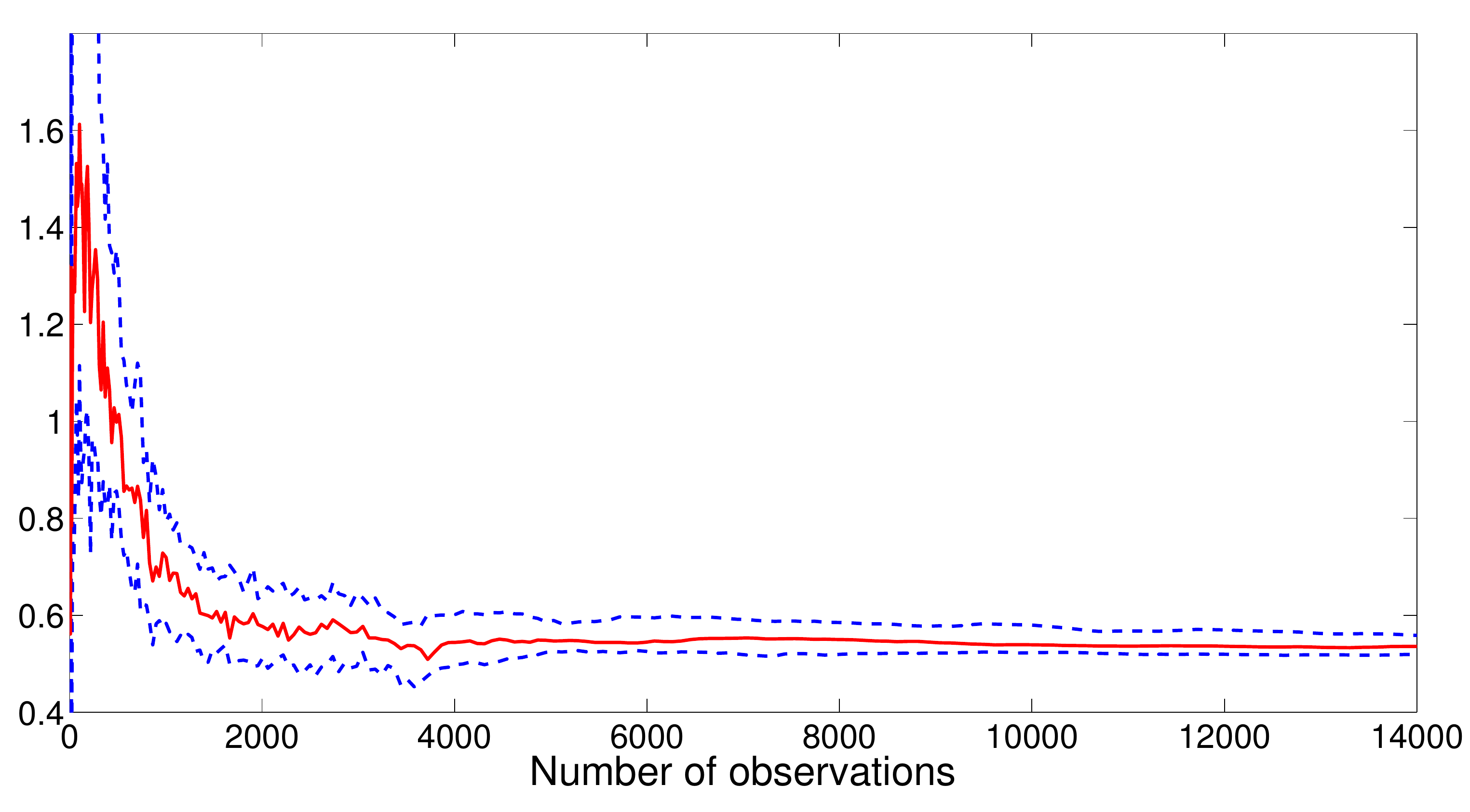}}
   \subfloat[The averaged OEM algorithm.]{\includegraphics[width=0.5\textwidth]{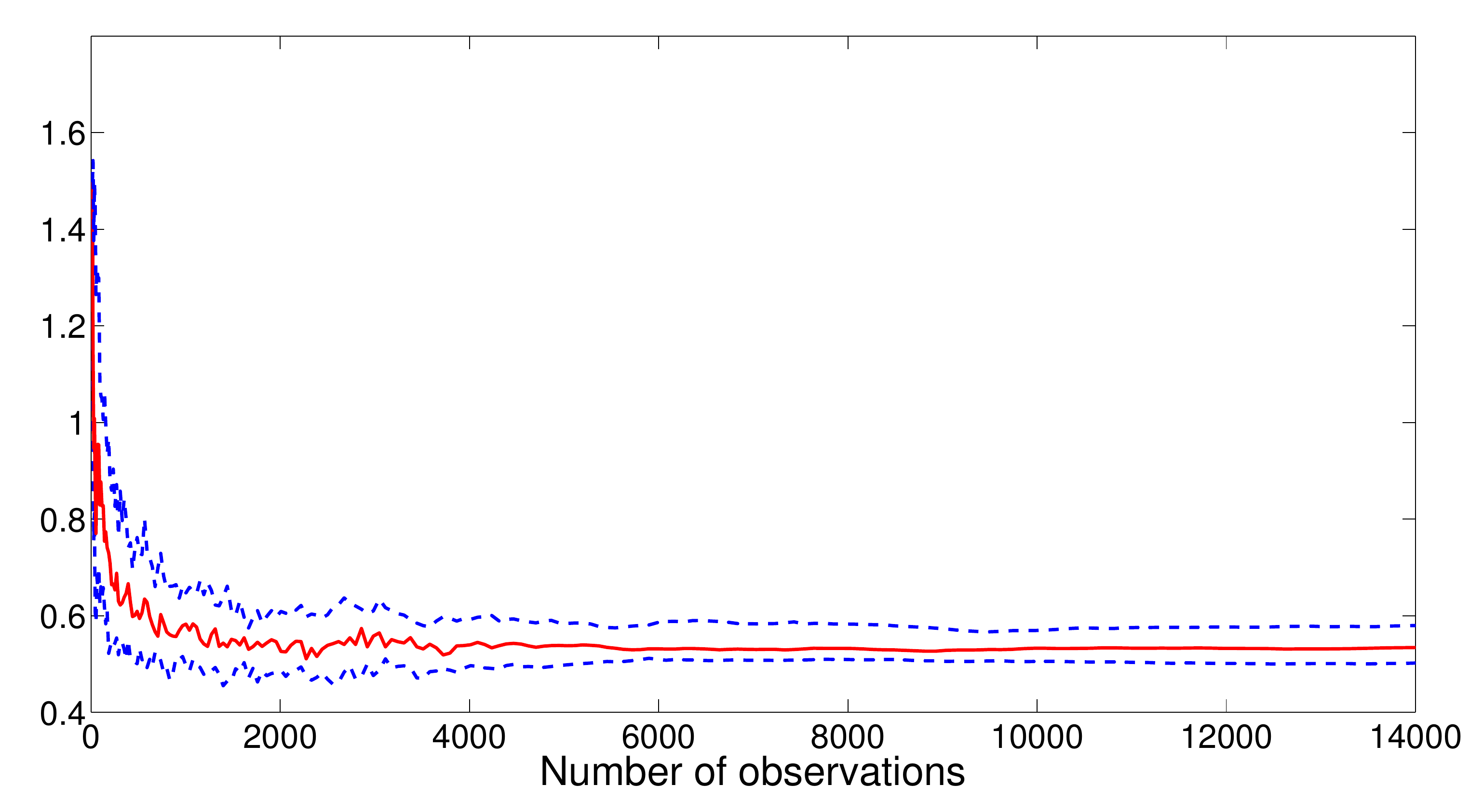}}
   \caption{Estimation of $v$ using the online EM and the BOEM algorithms (top)  and their averaged versions (bottom). Each plot displays the empirical median (bold line) and the first and last quartiles (dotted lines) over $100$ independent Monte Carlo runs with $\tau_{n} = n^{1.1}$ and $\gamma_{n} = n^{-0.53}$.}
   \label{BOEM:fig:quantilev}
 \end{figure}

 \begin{figure}[!h]
   \centering
   \subfloat[The averaged BOEM algorithm.]{\includegraphics[width=0.5\textwidth]{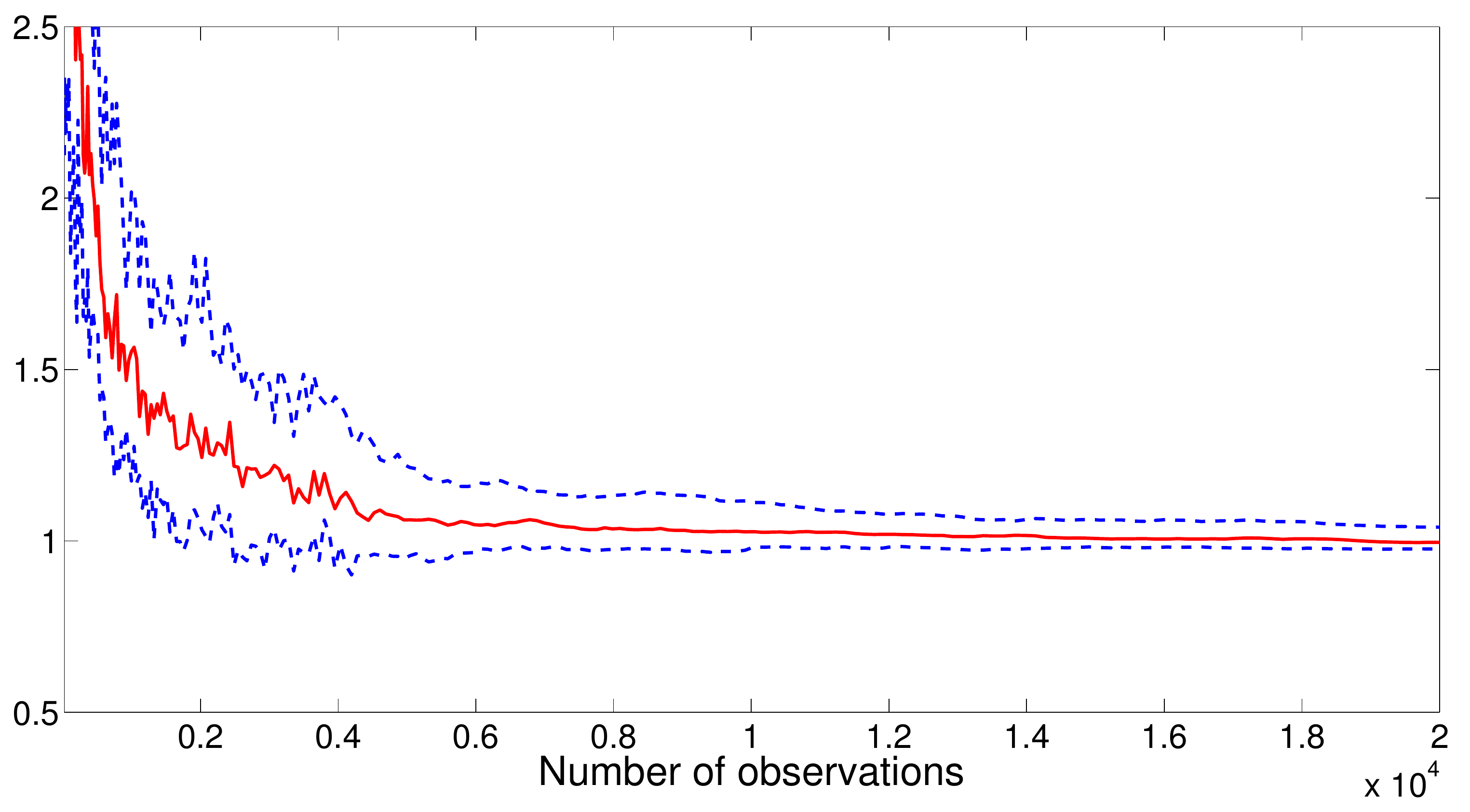}}
   \subfloat[The averaged OEM algorithm.]{\includegraphics[width=0.5\textwidth]{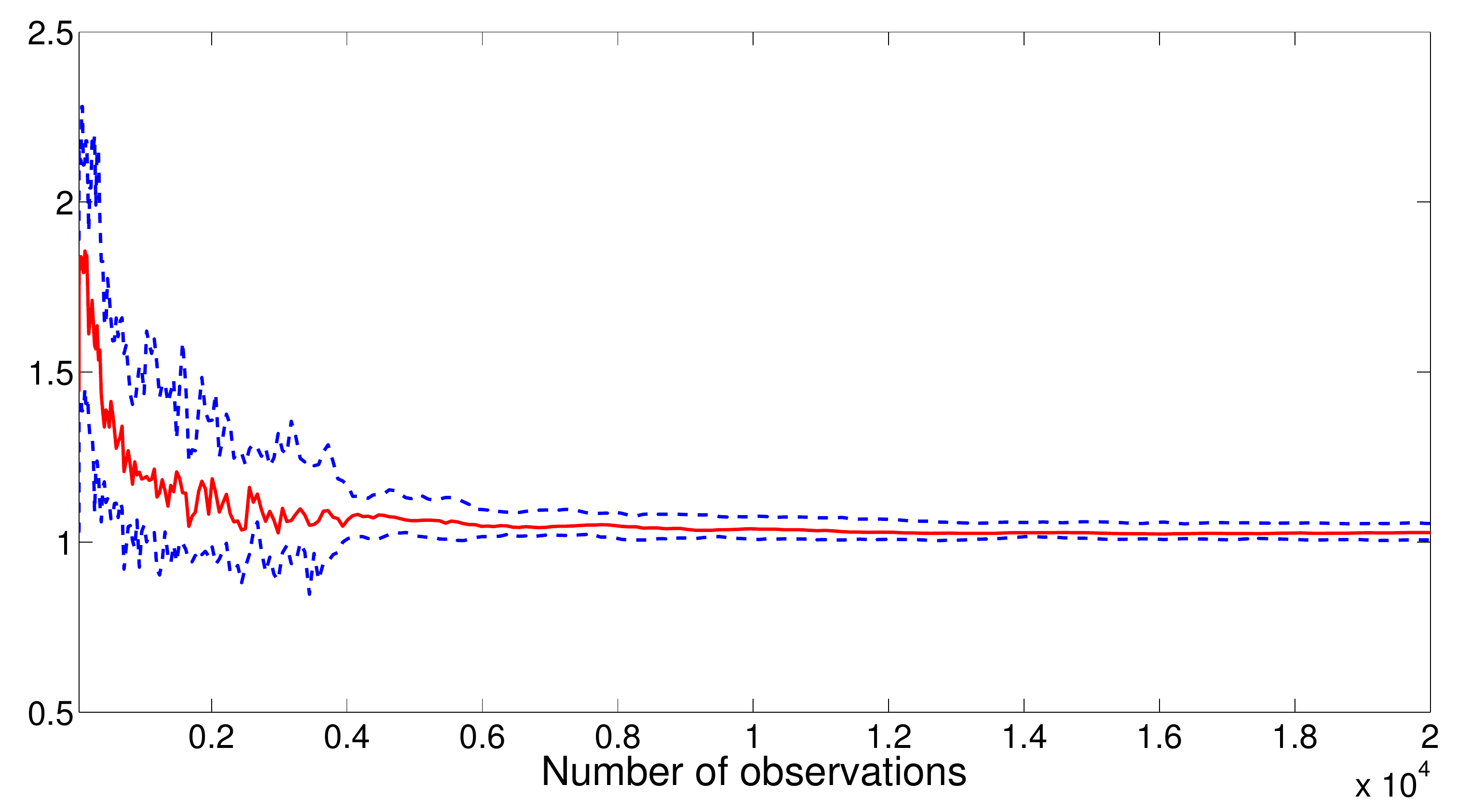}}
   \caption{Estimation of $x_{1}$ using the averaged OEM and the averaged BOEM algorithms. Each plot displays the empirical median (bold line) and the first and last quartiles (dotted lines) over $100$ independent Monte Carlo runs with $\tau_{n} = n^{1.1}$ and $\gamma_{n} = n^{-0.53}$. The first ten observations are omitted for a better visibility.}
   \label{BOEM:fig:quantilex1}
 \end{figure}

We now compare the averaged BOEM algorithm to a recursive maximum likelihood (RML) procedure (see \cite{legland:mevel:1997,tadic:2010}) combined with Polyak-Ruppert averaging (see \cite{polyak:juditsky:1992}). We want to estimate the variance $v$ and the transition matrix $m$. All the runs are started from $v = 2$ and from a matrix $m$ with each entry equal to $1/d$.  The RML algorithm  follows a stochastic approximation update and depends on a step-size sequence $\{\gamma_{n}\}_{n\geq 0}$ which is chosen in the same way as above.
Therefore, for a fair comparison, the RML algorithm (resp. the BOEM algorithm) is run with $\gamma_n= n^{-0.53}$ (resp.  $\tau_n = n^{1.1}$).  Figure~\ref{BOEM:fig:boxion} displays the empirical median and empirical first and last quartiles of the estimation of $m(1,1)$ as a function of the number of observations over $100$ independent Monte Carlo runs. For both algorithms, the bias and the variance of the estimation decrease as $n$ increases. Nevertheless, the bias and/or the variance of the averaged BOEM algorithm decrease faster than those of the averaged RML algorithm (similar graphs have been obtained for the estimation of the other entries of the matrix $m$ and for the estimation of $v$; see \cite[Section~$4$]{lecorff:fort:2011-supp}). As a conclusion, it is advocated to use the averaged BOEM algorithm instead of the averaged RML algorithm.
\begin{figure}[!h]
   \centering
   \subfloat[The averaged BOEM algorithm.]{\includegraphics[width=0.45\textwidth]{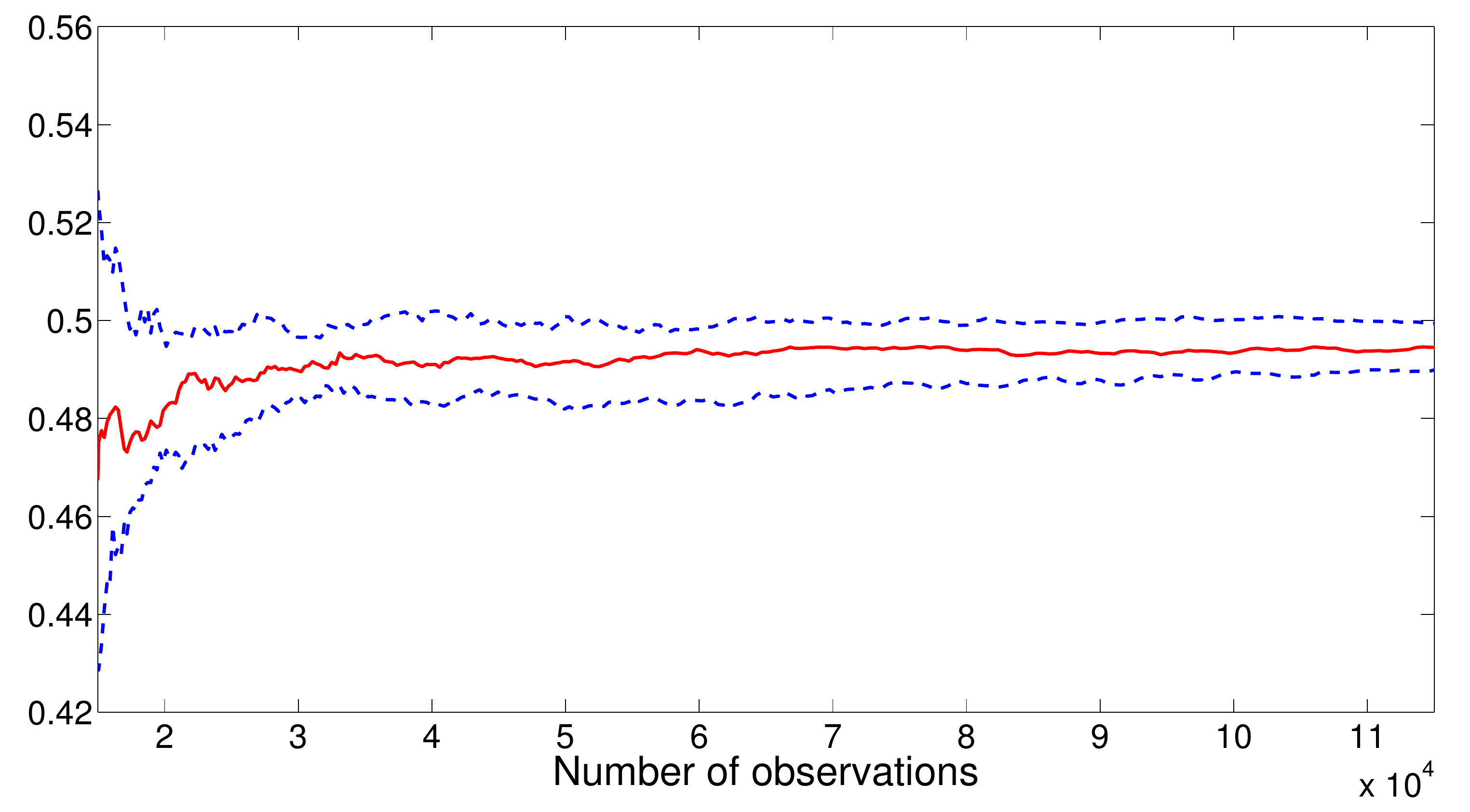}}
   \subfloat[The averaged RML algorithm.]{\includegraphics[width=0.45\textwidth]{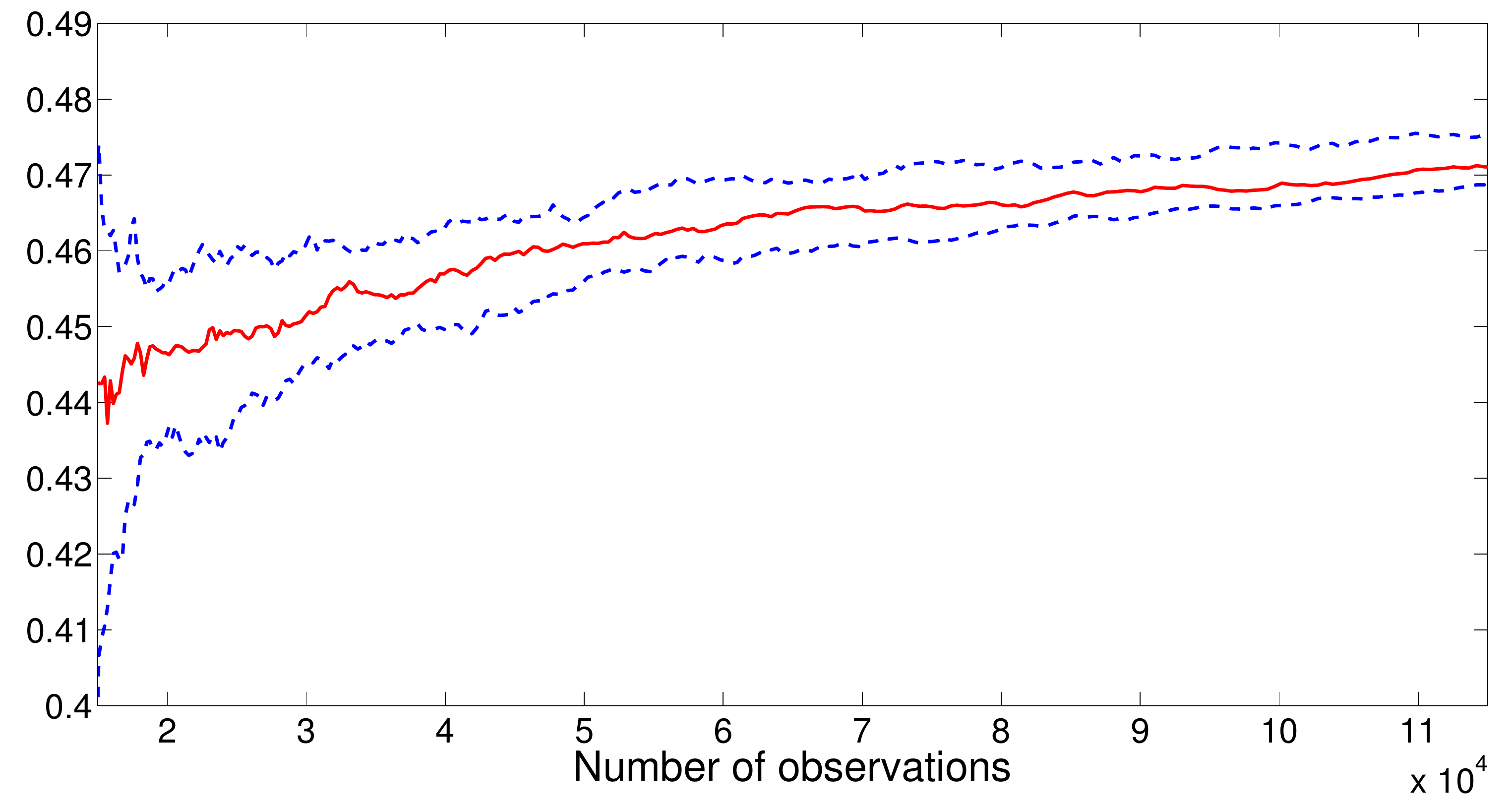}}
   \caption{Empirical median (bold line) and first and last quartiles (dotted line) for the estimation of $m(1,1)$ using the averaged RML algorithm (right)  and the averaged BOEM algorithm (left). The true values is $m(1,1)=0.5$ and the averaging procedure is starter after $10000$ observations. The first $10000$ observations are not displayed for a better clarity.}
   \label{BOEM:fig:boxion}
 \end{figure}

\section{Convergence of the Block Online EM algorithms}
\label{BOEM:sec:convergence}

\subsection{Assumptions}
Consider the following assumptions.
\begin{hypA}
\label{BOEM:assum:exp}
  \begin{enumerate}[(a)]
  \item \label{BOEM:assum:exp:decomp}There exist continuous functions $\phi :
    \paramset \to \Rset$, $\psi : \paramset \to \Rset^d$ and $S: \Xset \times
    \Xset \times \Yset \to \Rset^d$ s.t.
\begin{equation*}
  \label{BOEM:eq:exponential:family}
  \log m_\param(x,x') + \log g_\param(x',y) = \phi(\param) +
  \pscal{S(x,x,',y)}{\psi(\param)} \eqsp,
\end{equation*}
where $\pscal{\cdot}{\cdot}$ denotes the scalar product on $\Rset^d$.
\item \label{BOEM:assum:exp:convex} There exists an open subset $\Sset$ of $\Rset^d$
  that contains the convex hull of $S(\Xset \times \Xset \times \Yset)$.
\item \label{BOEM:assum:exp:max}There exists a continuous function $\bar \param :
  \Sset \to \paramset$ s.t. for any $s \in \Sset$,
\[
\bar \param(s) = \mathrm{argmax}_{\param \in \paramset} \; \left\{ \phi(\param) +
  \pscal{s}{\psi(\param)} \right\} \eqsp.
\]
  \end{enumerate}
\end{hypA}

\label{BOEM:sec:assum}
\begin{hypA}\label{BOEM:assum:strong}
  There exist $\sigma_{-}$ and $\sigma_{+}$ s.t. for any
  $\left(x,x^{\prime}\right)\in\Xset^2$ and any $\param \in \paramset$, $0
  <\sigma_{-} \leq m_\param(x,x^{\prime})\leq \sigma_{+}$.  Set $\rho \eqdef 1 -
  (\sigma_-/\sigma_+)\eqsp.$
\end{hypA}
A\ref{BOEM:assum:strong}, often referred to as the strong mixing condition, is commonly used to prove the forgetting property of the
initial condition of the filter, see e.g.
\cite{delmoral:guionnet:1998,delmoral:ledoux:miclo:2003}. This assumption holds
for example if $\Xset$ is finite or for linear state-spaces with truncated gaussian state and measurement noises.
More generally, this condition holds when $\Xset$ is compact. Note in addition that by \cite[Theorem~$16.0.2$]{meyn:tweedie:1993}, A\ref{BOEM:assum:strong} implies that the Markov kernel $M_\param$ has a unique invariant distribution which guarantees the existence of the unique invariant distribution $\pi_\param$ for $K_\param$.

We now introduce assumptions on the observation process $\bfY\eqdef\{Y_k\}_{k\in\Zset}$. It is defined on some probability space $(\Omega,\mathcal{F},\PPim)$. We stress that this process is not necessarily the observation of an HMM. Let
\begin{equation}
\label{BOEM:eq:Zfield}
\mathcal{F}_{k}^{\bfY} \eqdef \sigma\left(\{Y_{u}\}_{u\leq k}\right)\quad \mbox{and}\quad \mathcal{G}_{k}^{\bfY} \eqdef \sigma\left(\{Y_{u}\}_{u\geq k}\right)
\end{equation}
be $\sigma$-fields associated to $\bfY$. We also define the $\beta$-mixing coefficients
by, see \cite{davidson:1994},
\begin{equation}
\label{BOEM:eq:def:mixing}
\mix^{\bfY}(n)=\underset{u\in\Zset}{\sup}\,\underset{B\in\mathcal{G}_{u+n}^{\bfY}}{\sup}\,\CE[]{}{}{|\PP(B| \mathcal{F}_{u}^{\bfY}) - \PP(B)|}\eqsp, \forall\; n\geq0\eqsp.
\end{equation}

\begin{hypA}\label{BOEM:assum:moment:sup}
\hspace{-0.15cm}-($p$)  $\CE[]{}{}{\sup_{x,x' \in \Xset^2} \, |S(x,x',Y_{0})|^{p}}<+\infty$.
\end{hypA}
\begin{hypA}\label{BOEM:assum:obs}
\begin{enumerate}[(a)]
\item \label{BOEM:assum:obs:mix}$\bfY$ is a $\beta$-mixing stationary sequence such that there exist $C\in[0,1)$ and $\mix\in(0,1)$ satisfying, for any $n\geq 0$,
  $\mix^{\bfY}(n) \leq C\mix^{n}$, where $\mix^{\bfY}$ is defined in
  \eqref{BOEM:eq:def:mixing}.
  \item \label{BOEM:assum:obs:b+:b-} $\CE[]{}{}{|\log b_-(Y_0)| +|\log b_+(Y_0)| }<+\infty$ where
    \begin{eqnarray*}
      b_-(y) & \eqdef  \inf_{\param \in \paramset} \int g_\param(x,y) \lambda(\rmd x)  \eqsp, \\
  b_+(y) & \eqdef  \sup_{\param \in \paramset} \int g_\param(x,y) \lambda(\rmd x) \eqsp.
    \end{eqnarray*}
    \end{enumerate}
\end{hypA}
Upon noting that, for all $n\geq 0$, $\mix^{\bfY}(n)\leq \mix^{(\bfX,\bfY)}(n)$, we can prove that A\ref{BOEM:assum:obs}\eqref{BOEM:assum:obs:mix} holds when $\bfY$ is the observation process of a an HMM under classical geometric ergodicity conditions \cite[Chapter~$15$]{meyn:tweedie:1993} and \cite[Chapter~$14$]{cappe:moulines:ryden:2005}.

\begin{hypA}\label{BOEM:assum:size-block}
There exists $c>0$ and $a>1$ such that for all $n\ge 1$, $\tau_{n} = \lfloor cn^{a}\rfloor$.
\end{hypA}
For $p>0$ and $Z$ a random variable measurable w.r.t. the $\sigma$-algebra
$\sigma\left( Y_n, n\in \Zset\right)$, set $\lpnorm{Z}{p}\eqdef
\left(\CE{}{}{ |Z|^p}\right)^{1/p}$.
\begin{hypA}\hspace{-0.1cm}-($p$)\label{BOEM:assum:SMCapprox}
There exists $b\ge (a+1)/2a$ (where $a$ is defined in A\ref{BOEM:assum:size-block}) such that, for any $n\ge 0$,
  \[
  \lpnorm{S_{n} - \widetilde S_{n}}{p} = O(\tau_{n+1}^{-b})\eqsp,
  \]
  where $ \widetilde S_{n}$ is the Monte Carlo approximation of $S_{n}$ which is defined by \eqref{BOEM:eq:Bonem:recursion}.
\end{hypA}
A\ref{BOEM:assum:SMCapprox} gives a $\rmL_{p}$ control of the Monte Carlo error on each block. In \cite[Theorem~$1$]{dubarry:lecorff:2012}, such bounds are given for Sequential Monte Carlo algorithms. Practical conditions to ensure A\ref{BOEM:assum:SMCapprox} are given in \cite{lecorff:fort:2012} in the case of Sequential Monte Carlo methods.

\subsection{The limiting EM algorithm}
\label{BOEM:sec:BonEM:EM}
In the sequel, $\mathcal{M}(\Xset)$ denotes the set of all probability distributions on $(\Xset,\sigmaX)$.
\begin{theorem}\label{BOEM:th:LGN}
  Let $\bar p>2$. Assume that A\ref{BOEM:assum:exp}-\ref{BOEM:assum:strong}, A\ref{BOEM:assum:moment:sup}-($\bar p$) and
  A\ref{BOEM:assum:obs} hold.
  \begin{enumerate}[i)]
  \item \label{BOEM:th:LGN:phi}For any $\param\in\paramset$, there
  exists a r.v. $\limE{\param}{\bfY}$ s.t. 
  \begin{multline}
  \label{BOEM:th:LGN:smooth}
  \underset{\param\in\paramset,\,\chi\in\mathcal{M}(\Xset)}{\sup}\left|\CE[Y_{-\tau:\tau}]{\chi,-\tau-1}{\param}{S(X_{-1},X_{0},Y_{0})} - \limE{\param}{\bfY}\right| \\
  \leq C\rho^{\tau}\underset{(x,x')\in\Xset^{2}}{\sup}\left|S(x,x',Y_{0})\right|\eqsp,\quad
  \ps{\PPim}\eqsp,
   \end{multline}
   where $C$ is a finite constant.
    Define  for all $\param\in\paramset$,
\begin{equation}
\label{BOEM:eq:mapS}
\mapS(\param) \eqdef \CE[]{}{}{\limE{\param}{\bfY}}\eqsp.
\end{equation}
\item \label{BOEM:th:LGN:continuous}$\param \mapsto \mapS(\param)$ is continuous on $\paramset$ and for any $T>0$,
   \begin{equation}
   \label{BOEM:th:LGN:ergodic}
  \bar S_{\tau}^{\chi,T}(\param, \bfY)
    \underset{\tau\rightarrow
      +\infty}{\longrightarrow} \mapS(\param)\eqsp,\quad\ps{\PPim}\eqsp,
\end{equation}
where $\bar S_{\tau}^{\chi,T}(\param, \bfY)$ is defined by \eqref{BOEM:eq:rewrite:barS}.
\item  \label{BOEM:th:LGN:lp}Assume in addition that A\ref{BOEM:assum:SMCapprox}-($\bar p$) holds. For any $p\in(2,\bar p)$, there exists a constant $C$ s.t. for any $n\geq 1$,
\begin{equation*}
  \lpnorm{\widetilde S_{n}-\mapS(\param_n)}{p}
\leq  \frac{C}{\sqrt{\tau_{n+1}}}  \eqsp,
\end{equation*}
where $\widetilde S_{n}$ is the Monte Carlo approximation of $S_{n}$ defined by \eqref{BOEM:eq:Bonem:recursion}.
\end{enumerate}
\end{theorem}
Theorem~\ref{BOEM:th:LGN}  allows to introduce the limiting EM algorithm, defined as the
deterministic iterative algorithm $\check{\param}_n =
\limEMmapt(\check{\param}_{n-1})$ where
\begin{equation}
  \label{BOEM:eq:EM:algorithm}
\limEMmapt(\param) \eqdef  \bar \param \left( \mapS(\param)\right) \eqsp.
\end{equation}
The limiting EM can be seen as an EM algorithm applied as if the whole trajectory $\bfY$ was observed instead of $Y_{0:T}$. For this limiting EM, the so-called sufficient statistics depend on the observations only through the mean $\CE[]{}{}{\limE{\param}{\bfY}}$.
The stationary points of the limiting EM are defined as
\begin{equation}
\label{BOEM:eq:def:staset}
\Staset\eqdef \left\{\param\in\paramset;\; \limEMmapt(\param)=
  \param\right\}\eqsp.
  \end{equation}
We show that there exists a Lyapunov function $\lyap$ w.r.t. to the map $\limEMmapt$ and the set $\Staset$ {\em i.e.}, a continuous function $W$ satisfying the two conditions:
 \begin{enumerate}[(i)]
 \item \label{BOEM:prop:lyap:claim1}for all $\param\in\paramset$, $\lyap\circ\limEMmapt(\param)-\lyap(\param)\geq 0\eqsp,$
 \item \label{BOEM:prop:lyap:claim2}for all compact set $\mathcal{K}\subset
   \paramset\setminus\Staset$,
   $\inf_{\param\in\mathcal{K}}\left\{\lyap\circ\limEMmapt(\param)-\lyap(\param)\right\}>0\eqsp.$
\end{enumerate}
For such a function, the sequence $\{\lyap(\check\param_k)\}_{k\ge 0}$ is nondecreasing and $\{\check\param_k\}_{k\ge 0}$ converges to $\Staset$.

Define, for any $m\ge 0$, $\param\in\paramset$ and probability distribution $\chi$ on $(\Xset,\sigmaX)$,
\begin{multline*}
p_{\param}^{\chi}\left(Y_{1}\middle| Y_{-m:0}\right)\\
\eqdef\frac{\int \chi(\rmd x_{-m})g_{\param}(x_{-m},Y_{m})\prod_{i=-m+1}^{1}\left\{m_{\param}(x_{i-1},x_{i})g_{\param}(x_{i},Y_{i})\right\}\lambda(\rmd x_{-m+1:1})}{\int \chi(\rmd x_{-m})g_{\param}(x_{-m},Y_{m})\prod_{i=-m+1}^{0}\left\{m_{\param}(x_{i-1},x_{i})g_{\param}(x_{i},Y_{i})\right\}\lambda(\rmd x_{-m+1:0})}\eqsp.
\end{multline*}
By \cite[Lemma~$2$ and Proposition~$1$]{douc:moulines:ryden:2004}, under A\ref{BOEM:assum:exp}-\ref{BOEM:assum:obs}, for any $\param\in\paramset$, there exists a random variable $\log p_{\param}\left(Y_{1}\middle| Y_{-\infty:0}\right)$, such that for any probability distribution $\chi$ on $(\Xset,\sigmaX)$, $\log p_{\param}\left(Y_{1}\middle| Y_{-\infty:0}\right)$ is the a.s. limit of $\log p_{\param}^{\chi}\left(Y_{1}\middle| Y_{-m:0}\right)$ as $m\to +\infty$ and
\begin{equation}
\label{BOEM:eq:limitlik}
T^{-1}\ell_{\param,T}^{\chi}(\bfY)\underset{T\to+\infty}{\longrightarrow}\ell(\param)\eqdef\CE[]{}{}{\log p_{\param}\left(Y_{1}\middle| Y_{-\infty:0}\right)}\eqsp,\; \ps{\PPim}\eqsp,
\end{equation}
where $\ell_{\param,T}^{\chi}(\bfY)$ is the log-likelihood defined by \eqref{BOEM:eq:loglik}. The function $\param\mapsto \ell(\param)$ may be interpreted as the limiting log-likelihood. We consider the function $W$, given, for all $\param\in\paramset$, by
\begin{equation}
\label{BOEM:def:lyap}
\lyap(\param) \eqdef \exp\left\{\ell(\param)\right\}\eqsp.
\end{equation}
To identify the stationary points of the limiting EM algorithm as the stationary points of $\ell$, we introduce an additional assumption.
\begin{hypA}
\label{BOEM:assum:regularity-grad}
\begin{enumerate}[(a)]
\item \label{assum:regularity:diff} For any $y\in\Yset$ and for all $(x,x^{\prime})\in\Xset^{2}$,
  $\param\mapsto g_{\param}(x,y)$  and $\param\mapsto m_{\param}(x,x^{\prime})$ are continuously differentiable on
  $\paramset$.
\item  \label{BOEMsupp:assum:regularity:phi} $\mathbb{E}\left[\phi(\bfY_{0})\right]<+\infty$ where
\begin{equation*}
\phi(y) \eqdef \underset{\param\in\paramset}{\sup}\;\underset{(x,x^{\prime})\in\Xset^{2}}{\sup}\left|\nabla_{\param}\log m_{\param}(x,x^{\prime})+ \nabla_{\param}\log g_{\param}(x^{\prime},y)\right|\eqsp.
\end{equation*}
\end{enumerate}
\end{hypA}

\begin{proposition}\label{BOEM:prop:lyap}
  Assume that A\ref{BOEM:assum:exp}-\ref{BOEM:assum:strong}, A\ref{BOEM:assum:moment:sup}-$(1)$ and
  A\ref{BOEM:assum:obs} hold. Then, the function $\lyap$ given by \eqref{BOEM:def:lyap} is a Lyapunov function for $(\limEMmapt,\Staset)$. Assume in addition that A\ref{BOEM:assum:regularity-grad} holds. Then, $\param\mapsto \ell(\param)$ is continuously differentiable and
  \[
  \Staset =  \left\{\param\in\paramset;\; \limEMmapt(\param)=
  \param\right\} = \{\param\in\paramset;\;\nabla\ell(\param)=0\}\eqsp.
  \]
\end{proposition}
Proposition~\ref{BOEM:prop:lyap} is proved in Section~\ref{BOEM:proof:prop:lyap}.

\begin{remark}
In the case where $\{Y_{k}\}_{k\ge 0}$ is the observation process of the stationary HMM $\{(X_{k},Y_{k})\}_{k\ge 0}$ parameterized by $\param_{\star}\in\paramset$, we can build a two-sided stationary extension of this process to obtain a sequence of observations $\{Y_{k}\}_{k\in\Zset}$. Following \cite[Proposition~$3$]{douc:moulines:ryden:2004}, the quantity $ \ell(\param)$ can be written as
\begin{align*}
\ell(\param) &= \mathbb{E}_{\param_{\star}}\left[\lim_{m\to +\infty}\log p_{\param}(Y_{1}|Y_{-m:0})\right]\\
&=\lim_{m\to +\infty}\mathbb{E}_{\param_{\star}}\left[\log p_{\param}(Y_{1}|Y_{-m:0})\right]\\
&=\lim_{m\to +\infty}\mathbb{E}_{\param_{\star}}\left[\mathbb{E}_{\param_{\star}}\left[\log p_{\param}(Y_{1}|Y_{-m:0})\middle|Y_{-m:0}\right]\right]\eqsp,
\end{align*}
where $p_{\param}(Y_{1}|Y_{-m:0})$ is the conditional distribution under the stationary distribution.
Since
\[
\mathbb{E}_{\param_{\star}}\left[\log p_{\param_{\star}}(Y_{1}|Y_{-m:0})\middle|Y_{-m:0}\right]-\mathbb{E}_{\param_{\star}}\left[\log p_{\param}(Y_{1}|Y_{-m:0})\middle|Y_{-m:0}\right]
\]
is the Kullback-Leibler divergence between $p_{\param_{\star}}(Y_{1}|Y_{-m:0})$ and $p_{\param}(Y_{1}|Y_{-m:0})$, for any $\param\in\paramset$, $\ell(\param_{\star}) - \ell(\param)\ge 0$ and $\param_{\star}$ is a maximizer of $\param\mapsto \ell(\param)$. If in addition $\param_{\star}$ lies in the interior of $\paramset$, then $\param_{\star}\in\Staset$.
\end{remark}

The following proposition gives sufficient conditions for the convergence of the limiting EM algorithm and the Monte Carlo BOEM algorithm to the set
$\Staset$.
\begin{theorem}\label{BOEM:th:bonem:conv}
Let $\bar p>2$. Assume that A\ref{BOEM:assum:exp}-\ref{BOEM:assum:strong}, A\ref{BOEM:assum:moment:sup}-($\bar p$) and
  A\ref{BOEM:assum:obs} hold. Assume in addition that $\lyap(\Staset)$ has an empty
   interior. For any initial value $\check{\param}_{0}\in\paramset$ , there exists $w_\star$ s.t. $\{\check{\param}_{k}\}_{k\geq 0}$ converges to $\{\param\in\Staset;\;\lyap(\param)=w_{\star}\}$.
   If in addition A\ref{BOEM:assum:size-block} and A\ref{BOEM:assum:SMCapprox}-($\bar p$) hold, then the sequence $\{\param_{n}\}_{n\geq 0}$ converges $\ps{\PPim}$ to the same stationary points.
\end{theorem}
Theorem~\ref{BOEM:th:bonem:conv} is a direct application of Proposition~\ref{BOEM:prop:fort:moulines:2003} for the limiting EM algorithm. The proof for the Monte Carlo BOEM algorithm is detailed in Section~\ref{BOEM:proof:th:bonem:conv}.
By Sard's theorem if $\lyap$ is at least $d_{\param}$ (where $\paramset\subset \Rset^{d_{\param}}$) continuously differentiable, then $\lyap(\Staset)$ has Lebesgue measure $0$ and hence has an empty interior.



\section{Rate of convergence of the Block Online EM algorithms}
\label{BOEM:sec:averaging}
We address the rate of convergence of the Monte Carlo BOEM algorithms to a point $\param_{\star}\in\Staset$. It is assumed that
\begin{hypA}\label{BOEM:assum:stable:fixpoint}
  \begin{enumerate}[(a)]
  \item \label{BOEM:assum:stable:fixpoint:C2}$\mapS$ and $\bar\param$ are twice continuously differentiable on $\paramset$ and $\Sset$.
  \item \label{BOEM:assum:stable:fixpoint:norm}There exists $0<\gamma<1$ s.t. the spectral radius of $\nabla_{s}(\mapS\circ\bar\param)_{s=\mapS(\param_{\star})}$ is lower than $\gamma$.
      \end{enumerate}
 \end{hypA}

Hereafter, for any sequence of random variables $\{Z_{n}\}_{n\geq 0}$, write
$Z_{n} = O_{\rmL_{p}}(1)$ if
$\sup_{n}\CE[]{}{}{|Z_{n}|^{p}}<\infty$ and $Z_{n} = O_{\mathrm{a.s}}(1)$ if $\sup_{n}|Z_{n}|<+\infty$ $\ps{\PPim}$

\begin{theorem}
\label{BOEM:th:rate:nonaverage}
Let $\bar p>2$. Assume that A\ref{BOEM:assum:strong}, A\ref{BOEM:assum:moment:sup}-($\bar p$),
A\ref{BOEM:assum:obs}-\ref{BOEM:assum:size-block}, A\ref{BOEM:assum:SMCapprox}-($\bar p$) and A\ref{BOEM:assum:stable:fixpoint} hold. Then, for any $p\in(2,\bar p)$,
\begin{equation}
  \label{BOEM:eq:rate:nonaverage}
\sqrt{\tau_n} \ \left[\param_{n} - \param_{\star}\right]\1_{\lim_{n}\param_{n}=\param_{\star}} = O_{\rmL_{p}}(1)+  \frac{1}{\sqrt{\tau_n}}O_{\rmL_{p/2}}(1)  O_{\mathrm{a.s}}\left(1\right)\eqsp.
\end{equation}
\end{theorem}

In (\ref{BOEM:eq:rate:nonaverage}), the rate is a function of the number of updates (i.e. the number of iterations of the algorithm). Theorem~\ref{BOEM:th:rate:averaged} shows that the averaging procedure reduces the influence of the block-size schedule: the rate of convergence is proportional to $T_{n}^{1/2}$ i.e.  to the inverse of the square root of the
total number of observations up to iteration $n$.

\begin{theorem}
\label{BOEM:th:rate:averaged}
Let $\bar p>2$. Assume that A\ref{BOEM:assum:strong}, A\ref{BOEM:assum:moment:sup}-($\bar p$),
A\ref{BOEM:assum:obs}-\ref{BOEM:assum:size-block}, A\ref{BOEM:assum:SMCapprox}-($\bar p$) and A\ref{BOEM:assum:stable:fixpoint} hold. Then, for any $p\in(2,\bar p)$,
\begin{equation}
  \label{BOEM:eq:rate:averaged}
  \sqrt{T_n} \left[\widetilde\param_{n} - \param_{\star}\right]\1_{\lim_{n}\param_{n}=\param_{\star}} = O_{\rmL_{p}}(1)+ \frac{n}{\sqrt{T_n} } O_{\rmL_{p/2}}(1)O_{\mathrm{a.s}}\left(1\right)\eqsp.
\end{equation}
\end{theorem}
Theorems~\ref{BOEM:th:rate:nonaverage} and~\ref{BOEM:th:rate:averaged} give the rates of convergence as a function of the number of updates but they can also be studied as a function of the number of observations. Let $\{\param^{\text{int}}_k\}_{k\ge 0}$ (resp. $\{\widetilde\param^{\text{int}}_k\}_{k\ge 0}$) be such that, for any $k\ge 0$, $\param^{\text{int}}_k$ (resp. $\widetilde\param^{\text{int}}_k$) is the value $\param_n$ (resp. $\widetilde\param_n$), where $n$ is the only integer such that $k\in[T_n+1,T_{n+1}]$. The sequences $\{\param^{\text{int}}_k\}_{k\ge 0}$ and $\{\widetilde\param^{\text{int}}_k\}_{k\ge 0}$ are piecewise constant and their values are updated at times $\{T_n\}_{n\ge 1}$.

By Theorem~\ref{BOEM:th:rate:nonaverage}, the rate of convergence of $\{\param^{\text{int}}_k\}_{k\ge 0}$ is given (up to a multiplicative constant) by $k^{-a/(2(a+1))}$, where $a$ is given by A\ref{BOEM:assum:size-block}. This rates is slower than $k^{-1/2}$ and depends on the block-size sequence (through $a$). On the contrary, by Theorem~\ref{BOEM:th:rate:averaged}, the rate of convergence of $\{\widetilde\param^{\text{int}}_k\}_{k\ge 0}$ is given (up to a multiplicative constant) by $k^{-1/2}$, for any value of $a$. Therefore, this rate of convergence does not depend on the block-size sequence.

\section{Proofs}
\label{BOEM:sec:proofs}
Define, for any initial density $\chi$ on $(\Xset,\sigmaX)$, any $\param \in\paramset$, any $\bfy\in\Yset^{\Zset}$ and  any $r < s \leq t$,
\begin{multline}
\label{BOEM:eq:define-Phi}
\smoothfunc{\chi,r}{\param,s,t}(h,\bfy)  \\
\eqdef \frac{\int \chi(x_r) \{ \prod_{i=r}^{t-1} m_\param(x_i,
  x_{i+1})g_{\param}(x_{i+1},y_{i+1}) \} \, h(x_{s-1},x_{s},y_{s}) \,
  \lambda(\rmd x_{r:t})}{\int \chi(x_r) \{\prod_{i=r}^{t-1} m_\param(x_i,
  x_{i+1})g_{\param}(x_{i+1},y_{i+1}) \}\, \lambda(\rmd x_{r:t})}\eqsp,
\end{multline}
for any bounded function $h$ on $\Xset^{2}\times\Yset$. Then, the intermediate quantity of the Block online EM algorithm is (see \eqref{BOEM:eq:rewrite:barS}),
\begin{equation}
  \label{BOEM:eq:rewrite:barS:phi}
  \bar S_{\tau}^{\chi,T}(\param, \bfY) \eqdef
\frac{1}{\tau} \sum_{t=T+1}^{T+\tau}\smoothfunc{\chi,T}{\param,t,T+\tau}(S,\bfY)\eqsp.
\end{equation}
\begin{lemma}
\label{BOEM:lem:continuite:theta:Phi}
Assume A\ref{BOEM:assum:exp}-\ref{BOEM:assum:strong}.  Let
$\bfy\in\Yset^{\Zset}$ s.t. $\sup_{x,x'}|S(x,x',y_i)| < +\infty$ for
any $i \in \Zset$.  Then for any $r >0$ and any distribution $\chi$
on $(\Xset,\sigmaX)$, $\param \mapsto
\smoothfunc{\chi,-r}{\param,0,r}(S,\bfy)$ is continuous on $\Theta$.
\end{lemma}
\begin{proof}
  Set $K_\param(x,x',y) \eqdef m_\param(x,x') g_\param(x',y)$.  Let $r>0$ and
  $\chi$ be a distribution on $(\Xset,\sigmaX)$. By
  definition of $\smoothfunc{\chi,-r}{\param,0,r}(S,\bfy)$
  (see~\eqref{BOEM:eq:define-Phi}) we have to prove that
\[
\param\mapsto \int \chi(\rmd x_{-r}) \left( \prod_{i=-r}^{r-1}   K_\param(x_i,
x_{i+1},y_{i+1}) \right) h(x_{-1},x_0,y_0) \ \rmd \lambda(x_{-r+1:r})
\]
is continuous for $h(x,x',y) =1$ and $h(x,x',y) = S(x,x',y)$.  By
A\ref{BOEM:assum:exp}\eqref{BOEM:assum:exp:decomp}, the function $\param \mapsto
\prod_{i=-r}^{r-1} K_\param(x_i, x_{i+1},y_{i+1})
\,h(x_{-1},x_{0},y_{0})$ is continuous. In addition, under A\ref{BOEM:assum:exp},
for any $\param\in\paramset$,
\begin{multline*}
  \left|\prod_{i=-r}^{r-1} K_\param(x_i,
    x_{i+1},y_{i+1}) \,h(x_{-1},x_{0},y_{0})\right| \\
  = |h(x_{-1},x_{0},y_{0})|\exp\left(2r \phi(\param)+
    \pscal{\psi(\param)}{\sum_{i=-r}^{r-1} S(x_i, x_{i+1},y_{i+1})}\right)\eqsp.
\end{multline*}
Since $\paramset$ is compact,  by A\ref{BOEM:assum:exp}, there exist
constants $C_{1}$ and $C_{2}$ s.t. the supremum in $\param\in\paramset$ of this expression is bounded above by
\[
C_{1}\sup_{x,x'}|h(x,x^{\prime},y_{0})|
  \exp\left(C_{2}\sum_{i=-r}^{r-1} \sup_{x,x'}|S(x,x',y_{i+1})|\right)\eqsp.
  \]
  Since $\chi$ is a distribution and $\lambda$ is a finite measure, the continuity follows from the dominated convergence theorem.
\end{proof}

Let us introduce the following shorthand $S_{s}(x,x') \eqdef S(x,x',Y_s)$. Define the shift
operator $\shift$ onto $\Yset^\Zset$ by $(\shift  \bfy)_k = \bfy_{k+1}$
for any $k \in \Zset$; and by induction, define the $s$-iterated shift operator $\shift^{s+1}\bfy   = \shift( \shift^{s} \bfy)$, with the convention that $\shift^0$ is the identity operator. For a function $h$, define $\osc(h) \eqdef \sup_{z,z'} | h(z)-h(z')|$.

\subsection{Proof of Theorem~\ref{BOEM:th:LGN}}
\label{BOEM:proof:th:LGN}
The proof of Theorem~\ref{BOEM:th:LGN} relies on auxiliary results about the
forgetting properties of HMM. Most of them are really close to published
results and their proof is provided in the supplementary material
\cite[Section~$3$]{lecorff:fort:2011-supp}. The main
novelty is the forgetting property of the bivariate smoothing distribution.

\textit{Proof of \ref{BOEM:th:LGN:phi})}
Note that under A\ref{BOEM:assum:moment:sup}-(1), $\CE[]{}{}{\osc(S_{0})} <
+\infty$.  Under A\ref{BOEM:assum:strong}, Proposition~\ref{BOEM:prop:expforget}\eqref{BOEM:item:forget-limit} implies that for any
$\param \in \paramset$, there exists a r.v. $\limE{\param}{\bfY}$ s.t. for any $r < s \leq T$,
\begin{equation}
  \label{BOEM:eq:th:LGN:tool1}
  \underset{\param\in\paramset}{\sup}\;\left|\smoothfunc{\chi,r}{\param,s,T}\left(S,\bfY\right)-\limE{\param}{\shift^s\bfY}\right|\leq
\left(\rho^{T-s}+\rho^{s-r-1}\right)\osc(S_{s})\eqsp.
\end{equation}
This concludes the proof of \eqref{BOEM:th:LGN:smooth}.

\textit{Proof of \ref{BOEM:th:LGN:continuous})} We introduce the following decomposition: for all
$T>0$,
 \begin{equation*}
   \bar S_{\tau}^{\chi,T}(\param, \bfY) =\frac{1}{\tau}\sum_{t=1}^\tau\left[\limE{\param}{\shift^{t+T}\bfY} + \left\{\smoothfunc{\chi,0}{\param,t,\tau}\left(S,\shift^{T}\bfY\right)-\limE{\param}{\shift^{t+T}\bfY}\right\}\right]\eqsp,
\end{equation*}
upon noting that by \eqref{BOEM:eq:rewrite:barS:phi}, $ \bar S_{\tau}^{\chi,T}(\param,
\bfY) = \tau^{-1}
\sum_{t=1}^\tau\smoothfunc{\chi,0}{\param,t,\tau}\left(S,\shift^{T}\bfY\right)$.
By (\ref{BOEM:eq:define-Phi}), (\ref{BOEM:eq:th:LGN:tool1}) and
A\ref{BOEM:assum:moment:sup}-(1)
$\CE[]{}{}{\left|\limE{\param}{\bfY}\right|}<+\infty$. Under A\ref{BOEM:assum:obs}, the ergodic
theorem (see e.g.~\cite[Theorem 24.1, p.314]{billingsley:1987}) states that, for any fixed $T$,
\[
\underset{\tau\to\infty}{\lim}\;\frac{1}{\tau}\sum_{t=1}^\tau\limE{\param}{\shift^{t+T}\bfY} = \CE[]{}{}{\limE{\param}{\bfY}}\eqsp,\quad\ps{\PPim}
\]
By (\ref{BOEM:eq:th:LGN:tool1}),
\begin{equation}
\label{BOEM:eq:th:LGN:tool2}
\frac{1}{\tau}\sum_{t=1}^\tau\left|\smoothfunc{\chi,0}{\param,t,\tau}\left(S,\shift^{T}\bfY\right)-\limE{\param}{\shift^{t+T}\bfY}\right|
\leq
\frac{1}{\tau}\sum_{t=1}^\tau\left(\rho^{\tau-t}+\rho^{t-1}\right)\osc(S_{t+T})
\eqsp.
\end{equation}
Set $Z_{t} \eqdef \frac{1}{t}\sum_{s=1}^t \osc(S_{s+T})$ and $Z_{0} \eqdef 0$.
Then, by an Abel transform,
\begin{equation}
\label{BOEM:eq:th:LGN:abel}
\frac{1}{\tau}\sum_{t=1}^\tau\rho^{t-1}\osc(S_{t+T}) = \rho^{\tau-1}Z_{\tau}+ \frac{1-\rho}{\tau}\sum_{t=1}^{\tau-1}t\rho^{t-1}Z_{t}\eqsp.
\end{equation}
By A\ref{BOEM:assum:moment:sup}-(1) and A\ref{BOEM:assum:obs}, the ergodic theorem implies that $\lim_{\tau \to
  \infty} Z_{\tau}= \CE[]{}{}{\osc(S_{0})}$, $\ps{\PPim}$ Therefore,
$\limsup_{\tau}Z_{\tau} <\infty$, $\ps{\PPim}$ Since $\sum_{t\geq 1}t\rho^{t-1}
<\infty$, this implies that
$\tau^{-1}\sum_{t=1}^\tau\rho^{t-1}\osc(S_{t+T})\underset{\tau \to
  +\infty}{\longrightarrow} 0$, $\ps{\PPim}$
Similarly,
\[
\frac{1}{\tau}\sum_{t=1}^\tau\rho^{\tau-t}\osc(S_{t+T}) = Z_{\tau} -
(1-\rho)\sum_{t=1}^{\tau-1}\rho^{\tau-t-1}Z_{t} +
\frac{1-\rho}{\tau}\sum_{t=1}^{\tau-1}t\rho^{t-1}Z_{\tau-t}\eqsp.
\]
Using the same arguments as for the second term in \eqref{BOEM:eq:th:LGN:abel}, we can state that $\lim_{\tau \to \infty} \tau^{-1}\sum_{t=1}^{\tau-1}t\rho^{t-1}Z_{\tau-t} =0$, $\ps{\PPim}$
Furthermore,
\begin{multline*}
\left|\sum_{t=1}^{\tau-1}\frac{\rho^{\tau-t-1}}{1-\rho}Z_{t} - \CE[]{}{}{\osc(S_{0})}\right| \leq \sum_{t=1}^{\tau-1}\frac{\rho^{\tau-t-1}}{1-\rho}\left|Z_{t}-\CE[]{}{}{\osc(S_{0})}\right|\\
+ \;\CE[]{}{}{\osc(S_{0})}\rho^{\tau-1}\eqsp.
\end{multline*}
Since, $\ps{\PPim}$, $Z_{\tau}\underset{\tau \to +\infty}{\longrightarrow}
\CE[]{}{}{\osc(S_{0})}$, the RHS converges $\ps{\PPim}$ to
$0$ and
\[
\underset{\tau\to +\infty}{\lim}\left|Z_{\tau}-(1-\rho)\sum_{t=1}^{\tau-1}\rho^{\tau-t-1}Z_{t}\right| = 0\eqsp,\quad \ps{\PPim}
\]
Hence, the RHS in \eqref{BOEM:eq:th:LGN:tool2} converges $\ps{\PPim}$ to $0$ and
this concludes the proof of (\ref{BOEM:th:LGN:ergodic}).
We now prove that the function
$\param\mapsto\CE[]{}{}{\limE{\param}{\bfY}} $ is
continuous by application of the dominated convergence theorem.  By Proposition~\ref{BOEM:prop:expforget}\eqref{BOEM:item:forget-limit}, for any $\bfy$ s.t.
$\osc(S_{0}) < \infty$,
\[
\lim_{r \to +\infty} \sup_{\param \in \paramset}
\left|\smoothfunc{\chi,-r}{\param,0,r}(S,\bfy) -
  \limE{\param}{\bfy} \right|=0\eqsp.
\]
Then, by Lemma~\ref{BOEM:lem:continuite:theta:Phi}, $\param \mapsto
\limE{\param}{\bfy}$ is continuous for any $\bfy$ such
that $\osc(S_{0}) < +\infty$.  In addition, $\sup_{\param \in
  \paramset} \left|\limE{\param}{\bfY}\right| \leq
\sup_{x,x'}|S(x,x',Y_0)|$.  We then conclude by A\ref{BOEM:assum:moment:sup}-(1).

\textit{Proof of \ref{BOEM:th:LGN:lp})} Let $m_{n}, v_{n}$ be positive integers s.t. $1\leq m_{n}\leq \tau_{n+1}$ and
  $\tau_{n+1} = 2v_{n}m_{n}+r_{n}$, where $0\leq r_{n}< 2m_{n}$.  Set $\Delta p
  \eqdef p^{-1} - \bar p^{-1}$.  By the Minkowski inequality combined with
  Lemmas~\ref{BOEM:lem:lpcontrol:dec:rho}, \ref{BOEM:lem:lpcontrol:dec} applied with
  $q_{n}\eqdef 2v_{n}m_{n}$, there exists a constant $C$ s.t.
\[
\lpnorm{S_{n}-\mapS(\param_n)}{p}\leq C\left[ \rho^{m_{n}} +
  \frac{m_{n}}{\tau_{n+1}} + \mix^{m_{n}\Delta
    p}+\frac{1}{\sqrt{\tau_{n+1}}}\right]\eqsp.
\]
The proof is concluded by choosing $m_{n} = \lfloor - \log \tau_{n+1}/\left(\log \rho \vee \Delta p\log  \mix\right) \rfloor$ and by A\ref{BOEM:assum:SMCapprox}-($\bar p$) (since $b$ in A\ref{BOEM:assum:SMCapprox}-($\bar p$) is such that $b\ge 1/2$).

\subsection{Proof of Proposition~\ref{BOEM:prop:lyap}}
\label{BOEM:proof:prop:lyap}
$\,$\\
\textit{(Continuity of $\limEMmapt$ and $\lyap$)} By
A\ref{BOEM:assum:exp}\eqref{BOEM:assum:exp:max} and Theorem~\ref{BOEM:th:LGN}, the function
$\limEMmapt$ is continuous.  Under A\ref{BOEM:assum:exp}-\ref{BOEM:assum:strong} and
A\ref{BOEM:assum:obs}, there exists a continuous function $ \ell$ on
$\paramset$ s.t. $\lim_T T^{-1}\loglik{\chi}{\param,T}(\bfY) =
\ell(\param)$ $\ps{\PPim}$ for any distribution $\chi$ on
$(\Xset,\sigmaX)$ and any $\param\in\paramset$, (see \cite[Lemma~$2$ and Propositions~$1$ and~$2$]{douc:moulines:ryden:2004}, see also
\cite[Theorem~$3.8$]{lecorff:fort:2011-supp}). Therefore, $\lyap$
is continuous.

\medskip

\textit{Proof of Proposition~\ref{BOEM:prop:lyap}~\eqref{BOEM:prop:lyap:claim1}}
Under Assumption A\ref{BOEM:assum:exp}\eqref{BOEM:assum:exp:decomp}
\[
\frac{1}{T}\log p_{\param}(x_{0:T},Y_{1:T})=\phi(\param) + \pscal{\left\{\frac{1}{T}\sum_{t=1}^{T}S(x_{t-1},x_{t},Y_{t})\right\}}{\psi(\param)}\eqsp,
\]
where $ p_{\param}(x_{0:T},Y_{1:T})$ is defined by \eqref{BOEM:eq:completelik}.
Upon noting that
\begin{multline*}
\int S(x_{t-1},x_{t},Y_{t})\frac{p_{\param}(x_{0:T},Y_{1:T})}{\int
  p_{\param}(z_{0:T},Y_{1:T})\lambda(\rmd z_{1:T})\chi(\rmd z_0)
}\lambda(\rmd x_{1:T}) \chi(\rmd x_0) \\
= \smoothfunc{\chi,0}{\param,t,T}(S,\bfY)\eqsp,
\end{multline*}
the Jensen inequality gives, $\ps{\PPim}$,
\begin{multline}
\label{BOEM:eq:loglik-diff}
\frac{1}{T}\loglik{\chi}{\limEMmapt(\param),T}(\bfY)- \frac{1}{T}\loglik{\chi}{\param,T}(\bfY) \geq \phi(\limEMmapt(\param)) + \pscal{\frac{1}{T}\sum_{t=1}^{T}\smoothfunc{\chi,0}{\param,t,T}(S,\bfY)}{\psi(\limEMmapt(\param))}\\
-\phi(\param) -
\pscal{\frac{1}{T}\sum_{t=1}^{T}\smoothfunc{\chi,0}{\param,t,T}(S,\bfY)}{\psi(\param)}\eqsp.
\end{multline}
Under A\ref{BOEM:assum:exp}-\ref{BOEM:assum:obs}, it holds by Theorem~\ref{BOEM:th:LGN} and
\cite[Lemma $2$ and Proposition $1$]{douc:moulines:ryden:2004} (see also
\cite[Theorem~$3.8$]{lecorff:fort:2011-supp}) that
for all $\param\in\paramset$, $\ps{\PPim}$,
\begin{equation*}
  \frac{1}{T}\sum_{t=1}^{T}\smoothfunc{\chi,0}{\param,t,T}(S,\bfY)\limT{T}\mapS(\param)\eqsp,  \qquad
  \frac{1}{T}\loglik{\chi}{\param,T}(\bfY)\limT{T}
 \ln  \lyap(\param)\label{BOEM:eq:limloglik1}\eqsp.
\end{equation*}
Therefore, when $T\to +\infty$, \eqref{BOEM:eq:loglik-diff} implies
\begin{equation}
  \label{BOEM:eq:lyap-diff}
  \ln \left(\lyap(\limEMmapt(\param))/ \lyap(\param)\right)\geq \phi(\limEMmapt(\param)) + \pscal{\mapS(\param)}{\psi(\limEMmapt(\param))}-\phi(\param) - \pscal{\mapS(\param)}{\psi(\param)}\eqsp.
\end{equation}
By definition of $\bar\param$ and $\limEMmapt$ (see
A\ref{BOEM:assum:exp}\eqref{BOEM:assum:exp:max} and (\ref{BOEM:eq:EM:algorithm})), the RHS is
non negative. This concludes the proof of
Proposition~\ref{BOEM:prop:lyap}\eqref{BOEM:prop:lyap:claim1}.

\medskip

\textit{Proof of Proposition~\ref{BOEM:prop:lyap}~\eqref{BOEM:prop:lyap:claim2}} We prove
that $\lyap\circ\limEMmapt(\param) - \lyap(\param) = 0$ if and only if
$\param\in\Staset$. Since $\lyap\circ\limEMmapt - \lyap$ is continuous, this
implies that
$\underset{\param\in\K}{\inf}\lyap\circ\limEMmapt(\param)-\lyap(\param)>0$ for
all compact set $\K\subset \paramset\setminus \Staset$.  Let
$\param\in\paramset$ be s.t.  $\lyap\circ\limEMmapt(\param)-\lyap(\param)= 0$.
Then, the RHS in \eqref{BOEM:eq:lyap-diff} is equal to zero.  By definition of
$\bar\param$, $\limEMmapt(\param) = \param$ and thus $\param\in\Staset$.  The
converse implication is immediate from the definition of $\Staset$.

\medskip

\textit{Stationary points} If in addition A\ref{BOEM:assum:regularity-grad} holds, \cite[Theorem~$3.12$]{lecorff:fort:2011-supp} proves that, for any initial distribution $\chi$ on $(\Xset,\sigmaX)$,
\[
\frac{1}{T}\nabla_{\param}  \ell^{\chi}_{\param,T}(\bfY) \underset{T\rightarrow +\infty}{\longrightarrow}\nabla_{\param}\ell(\param)\quad\ps{\PPim}
\]
Therefore,
\[
\frac{1}{T}\nabla_{\param}\ell^{\chi}_{\param,T}(\bfY)=\nabla_{\param}\phi(\param) + \nabla_{\param}\psi^{'}(\param)\left\{\frac{1}{T}\sum_{t=1}^{T}\smoothfunc{\chi,0}{\param,t,T}(S,\bfY)\right\}\eqsp,
\]
where $A'$ is the transpose matrix of $A$. Theorem~\ref{BOEM:th:LGN} yield, $\ps{\PPim}$,
\[
\nabla_{\param}\ell(\param) = \nabla_{\param}\phi(\param) + \nabla_{\param}\psi^{'}(\param)\bar S(\param)\eqsp.
\]

The proof follows upon noting that by definition of $\bar\param$, the unique solution to the equation $\nabla_{\param}\phi(\tau) + \nabla_{\param}\psi^{'}(\tau)\bar S(\param)=0$ is $\tau = \limEMmapt(\param)$.

\subsection{Proof of Theorem~\ref{BOEM:th:bonem:conv}}
\label{BOEM:proof:th:bonem:conv}
The proof of Theorem~\ref{BOEM:th:bonem:conv} relies on Proposition~\ref{BOEM:prop:fort:moulines:2003} applied with $T(\param) \eqdef \limEMmapt(\param)$ and with $\param_{n+1} = \bar\param\left(\widetilde S_{\tau_{n+1}}^{\chi_{n},T_{n}}(\param_{n}, \bfY)\right)$. The key ingredient for this proof is the
control of the $\rmL_{p}$-mean error between the Monte Carlo Block Online EM algorithm and
the limiting EM. The proof of this bound is derived in
Theorem~\ref{BOEM:th:LGN} and relies on preliminary lemmas given in Appendix~\ref{BOEM:sec:tech:res}. The proof of \eqref{BOEM:eq:fort:moulines:2003} is now close to the proof of \cite[Proposition~$11$]{fort:moulines:2003} and is postponed to the supplement paper \cite[Section~$2.1$]{lecorff:fort:2011-supp}.

\subsection{Proof of Theorem~\ref{BOEM:th:rate:nonaverage}}
\label{BOEM:sec:proof:averaging}
Define $s_{\star}\eqdef \mapS(\param_{\star})$ and write
\begin{equation}
\label{BOEM:eq:rate:param}
\bar\param(\widetilde S_{n}) -  \bar\param(s_{\star})=\Upsilon(\widetilde S_{n}-s_{\star}) + \bar\param(\widetilde S_{n}) -  \bar\param(s_{\star}) - \Upsilon(\widetilde S_{n}-s_{\star})\eqsp,
\end{equation}
where $\Upsilon\eqdef\nabla\bar\param(s_{\star})$. We now derive the rate of convergence of the quantity $\widetilde S_{n}-s_{\star}$. Set $\limEMmap(s) \eqdef \mapS\circ\bar\param(s)$. Note that under A\ref{BOEM:assum:stable:fixpoint}\eqref{BOEM:assum:stable:fixpoint:norm}, $\rho(\Gamma)\leq\gamma$, where $\Gamma\eqdef\nabla\limEMmap(s_{\star})$. Since $\limEMmap(s_{\star}) =s_\star$, we write
\[
\widetilde S_{n}-s_{\star} =\jacG\left(\widetilde S_{n-1}-s_{\star}\right) + \widetilde S_{n}
-\limEMmap(\widetilde S_{n-1}) + \limEMmap(\widetilde S_{n-1})-
\limEMmap(s_{\star})-\jacG\left(\widetilde S_{n-1}-s_{\star}\right) \eqsp.
\]
Define $\{\mu_{n}\}_{n\geq 0}$ and $\{\rho_{n}\}_{n\geq 0}$ s.t. $\mu_{0} = 0$,
$\rho_{0} = \widetilde S_{0}-s_{\star}$ and
\begin{equation}
  \mu_{n} \eqdef \jacG\mu_{n-1} + e_n \eqsp, \qquad \rho_{n} \eqdef
  \widetilde S_{n}-s_{\star} - \mu_{n} \label{BOEM:eq:def:mu} \eqsp,  \qquad n\geq 1 \eqsp,
\end{equation}
where,
\begin{equation}
\label{BOEM:eq:define:epsilon}
e_{n}\eqdef \widetilde S_{n}-\mapS(\param_n)\eqsp, \qquad  n\geq 1 \eqsp.
\end{equation}

\begin{proposition}
\label{BOEM:prop:mu:rho:rate}
Assume A\ref{BOEM:assum:strong}, A\ref{BOEM:assum:moment:sup}-($\bar p$),
A\ref{BOEM:assum:obs}-\ref{BOEM:assum:size-block}, A\ref{BOEM:assum:SMCapprox}-($\bar p$) and A\ref{BOEM:assum:stable:fixpoint} for some $\bar
p>2$. Then for any $p \in (2,\bar p)$,
\[
\sqrt{\tau_n} \mu_n = O_{\rmL_{p}}(1)\quad\mbox{and}\quad \tau_n \rho_n \1_{\lim_{n}\param_{n}=\param_{\star}} = O_{\rmL_{p/2}}(1)O_{\mathrm{a.s}}(1)\eqsp.
\]
\end{proposition}

The proof of Proposition~\ref{BOEM:prop:mu:rho:rate} is on the same lines as the
proof of \cite[Theorem 6]{fort:moulines:2003}. The main ingredient is the
control of $\lpnorm{\mu_n}{p}$ which is a consequence of \cite[Result
178, p. 39]{polya:1976} and Theorem~\ref{BOEM:th:LGN}.  The detailed
proof is thus omitted and postponed to the supplementary
material~\cite[Section~$2.2$]{lecorff:fort:2011-supp}.

By Proposition~\ref{BOEM:prop:mu:rho:rate}, the first term in \eqref{BOEM:eq:rate:param} gives
\[
\sqrt{\tau_{n}}\Upsilon(S_{n}-s_{\star})\1_{\lim_{n} \param_{n} =\param_{\star}}  =   O_{\rmL_{p}}(1)+ \frac{1}{\sqrt{\tau_n} } O_{\rmL_{p/2}}(1)O_{\mathrm{a.s}}\left(1\right) \eqsp.
\]
A Taylor expansion with integral remainder term gives the rate of convergence of the second term. This concludes the proof of Theorem~\ref{BOEM:th:rate:nonaverage},  Eq.~(\ref{BOEM:eq:rate:nonaverage}).

\subsection{Proof of Theorem~\ref{BOEM:th:rate:averaged}}
In the sequel, for all function $\Xi$ on $\paramset\times\Yset^{\Zset}$ and all
$\upsilon\in\paramset$, we denote by
$\CE[]{}{}{\Xi(\param,\bfY)}_{\param=\upsilon}$ the function
$\param\mapsto \CE[]{}{}{\Xi(\param,\bfY)}$ evaluated at
$\param=\upsilon$. We preface the proof by the following lemma.
\begin{lemma}
\label{BOEM:lem:mu:rho:average:rate}
Assume A\ref{BOEM:assum:strong}, A\ref{BOEM:assum:moment:sup}-($\bar p$),
A\ref{BOEM:assum:obs}-\ref{BOEM:assum:size-block}, A\ref{BOEM:assum:SMCapprox}-($\bar p$) and A\ref{BOEM:assum:stable:fixpoint} for some $ \bar
p>2$. For any $p\in(2,\bar p)$,
\[
\underset{n\to
  +\infty}{\limsup}\; \frac{1}{\sqrt{T_{n+1}}}\;\lpnorm{\sum_{k=1}^{n}\tau_{k+1}e_{k}}{p}<\infty\eqsp,
\]
where $e_{n}$ is given by \eqref{BOEM:eq:define:epsilon}.
\end{lemma}
\begin{proof}
  By A\ref{BOEM:assum:size-block} and A\ref{BOEM:assum:SMCapprox}-($\bar p$), we have
  \[
\underset{n\to
  +\infty}{\limsup}\; \frac{1}{\sqrt{T_{n+1}}}\;\sum_{k=1}^{n}\tau_{k+1}\lpnorm{\widetilde S_{k}-S_{k}}{p}<\infty\eqsp.
  \]
  Then, it is sufficient to prove that
  \[
\underset{n\to
  +\infty}{\limsup}\; \frac{1}{\sqrt{T_{n+1}}}\;\lpnorm{\sum_{k=1}^{n}\tau_{k+1}\left(\bar S(\param_{k})-S_{k}\right)}{p}<\infty\eqsp.
  \]
  Let $p\in(2,\bar p)$. In the sequel, $C$ is a constant independent on $n$
  and whose value may change upon each appearance. Let $1\leq m_{n}\leq
  \tau_{n+1}$ and set $v_{n}\eqdef
  \left\lfloor\frac{\tau_{n+1}}{2m_{n}}\right\rfloor$. By
  Lemma~\ref{BOEM:lem:lpcontrol:dec} applied with $q_{k} \eqdef 2v_{k}m_{k}$, we
  have,
\begin{multline*}
\lpnorm{\sum_{k=1}^{n}\tau_{k+1}\left(\bar S(\param_{k})-S_{k}\right)}{p}\\ \leq C \left(\sum_{k=1}^{n} \{
  \tau_{k+1}\rho^{m_{k}} + m_{k} \}
  +\lpnorm{\sum_{k=1}^{n}\{ \delta_k + \zeta_k \}}{p}\right)\eqsp,
\end{multline*}
where $\delta_{k}$ and $\zeta_{k}$ are defined by
\begin{align*}
  \delta_{k}&\eqdef\sum_{t=2m_k}^{2v_{k}m_{k}} \left\{
   F_{t,k}(\param_{k},\bfY)   -
    \CE[\widetilde{\mathcal{F}}_{T_{k}}^{\bfY}]{}{}{F_{t,k}(\param_{k},\bfY) }  \right\}\eqsp,\\
  \zeta_{k}&\eqdef\sum_{t=2m_k}^{2v_{k}m_{k}} \left\{
    \CE[\widetilde{\mathcal{F}}_{T_{k}}^{\bfY}]{}{}{F_{t,k}(\param_{k},\bfY) } -
    \CE[]{}{}{\smoothfunc{\chi,-m_{k}}{\param,0,m_{k}}(S,\bfY)}_{\param =
      \param_k} \right\}
\end{align*}
and where $F_{t,k}(\param_{k},\bfY) \eqdef
  \smoothfunc{\chi,t-m_k}{\param_{k},t,t+m_k}(S,\shift^{T_k}  \bfY)$ and $\widetilde{\mathcal{F}}_{T_{k}}^{\bfY}$ is given by (\ref{BOEM:eq:Zfield:tilde}). We will prove below that there exists $C$ s.t.
\begin{align}
  \label{BOEM:lem:mu:rho:average:rate:tool1}
& \lpnorm{\zeta_k}{p} \leq C \, \mix^{m_{k}/pb} \tau_{k+1} \eqsp,  \qquad  \forall k \geq 1\\
 \label{BOEM:lem:mu:rho:average:rate:tool2}
& \lpnorm{ \sum_{k=1}^n \delta_k}{p} \leq C \sqrt{T_{n+1}} + C
\sum_{k=1}^n \tau_{k+1}\beta^{m_k /pb}\eqsp,  \qquad  \forall n \geq 1
\end{align}
so that the proof is concluded by choosing $m_{k} = \lfloor \eta\log\tau_{k+1}
\rfloor$, $\eta\eqdef \left(-1/\log \rho \right)\vee \left(- pb / \log
  \beta\right)$ and by using A\ref{BOEM:assum:size-block}.

We turn to the proof of (\ref{BOEM:lem:mu:rho:average:rate:tool1}). By the Berbee
Lemma (see \cite[Chapter $5$]{rio:1999}) and A\ref{BOEM:assum:obs}, there
exist $C\in [0,1)$ and $\mix\in (0,1)$ s.t. for all $k\geq 1$, there exists a
random variable $Y^{\star,(k)}_{T_{k}+m_{k}:T_{k+1}+m_{k}}$ on
$(\Omega,\mathcal{F},\PPim)$ independent from $\widetilde{\mathcal{F}}_{T_{k}}^{\bfY}$ with
the same distribution as $Y_{T_{k}+m_{k}:T_{k+1}+m_{k}}$ and
\begin{equation}
\label{BOEM:eq:berbee:mix}
\PP_{}\left\{Y^{\star,(k)}_{T_{k}+m_{k}:T_{k+1}+m_{k}}\neq Y_{T_{k}+m_{k}:T_{k+1}+m_{k}}\right\} \leq C\mix^{m_{k}}\eqsp.
\end{equation}
 Upon
noting that $
\CE[\widetilde{\mathcal{F}}_{T_{k}}^{\bfY}]{}{}{F_{t,k}(\param_{k},\bfY^{\star,(k)})}
=\CE[]{}{}{F_{t,k}(\param,\bfY)}_{\param = \param_k}$, we have
\begin{equation}
  \label{BOEM:eq:lem:mu:rho:average:rate:tool3}
  \zeta_{k} = \sum_{t=2m_k}^{2v_{k}m_{k}} \left\{
  \CE[\widetilde{\mathcal{F}}_{T_{k}}^{\bfY}]{}{}{F_{t,k}(\param_{k},\bfY)} -
  \CE[\widetilde{\mathcal{F}}_{T_{k}}^{\bfY}]{}{}{F_{t,k}(\param_{k},\bfY^{\star,(k)})}
\right\} \eqsp.
\end{equation}
Therefore, by setting $\mathcal{A}_k \eqdef
\{Y^{\star,(k)}_{T_{k}+m_{k}:T_{k+1}+m_{k}}\neq Y_{T_{k}+m_{k}:T_{k+1}+m_{k}}\}$,
      \begin{align*}
        \left|\zeta_{k}\right|
        &\leq \sum_{t=2m_k}^{2v_{k}m_{k}}
        \CE[\widetilde{\mathcal{F}}_{T_{k}}^{\bfY}]{}{}{\underset{\param\in\paramset}{\sup}\left|F_{t,k}(\param,\bfY)-F_{t,k}(\param,\bfY^{\star,(k)})\right|\1_{
            \mathcal{A}_k }}\eqsp.
      \end{align*}
      Minkowski and Holder (with $a\eqdef
      \bar{p}/p$ and $b^{-1}\eqdef 1-a^{-1}$) inequalities, combined with
      \eqref{BOEM:eq:berbee:mix}, A\ref{BOEM:assum:obs}, Lemma~\ref{BOEM:lem:bound-phi} and A\ref{BOEM:assum:moment:sup}-($\bar p$)
      yield (\ref{BOEM:lem:mu:rho:average:rate:tool1}).

      We now prove (\ref{BOEM:lem:mu:rho:average:rate:tool2}).  Upon noting that
      $\delta_{k}$ is $\widetilde{\mathcal{F}}_{T_{k+1}}^{\bfY}$-measurable and
      $\delta_{k}$ is a martingale increment, the Rosenthal inequality
      (see~\cite[Theorem 2.12, p.23]{hall:heyde:1980}) states that
      $\lpnorm{\sum_{k=1}^{n}\delta_{k}}{p} \leq C \left(\sum_{k=1}^n
        I_k^{(1)} \right)^{1/p} + C I_n^{(2)} $ where
         \begin{equation*}
      I_k^{(1)}  \eqdef \CE[]{}{}{\left|\delta_{k}\right|^{p}}\quad\mbox{and}\quad I_n^{(2)}  \eqdef \lpnorm{\left(\sum_{k=1}^{n}\CE[\widetilde{\mathcal{F}}_{T_{k}}^{\bfY}]{}{}{\left|\delta_{k}\right|^{2}}\right)^{1/2}}{p}\eqsp.
      \end{equation*}
      Using again $\CE[\widetilde{\mathcal{F}}_{T_{k}}^{\bfY}]{}{}{F_{t,k}(\param_{k},\bfY^{\star,(k)})}
      =\CE[]{}{}{F_{t,k}(\param,\bfY)}_{\param = \param_k}$ and (\ref{BOEM:eq:lem:mu:rho:average:rate:tool3})
      \[
      I_k^{(1)} \leq
            C\;\lpnorm{\sum_{t=2m_k}^{2v_{k}m_{k}} \left\{F_{t,k}(\param_{k},\bfY) -
          \CE[]{}{}{F_{t,k}(\param,\bfY)}_{\param=\param_{k}}
        \right\}}{p}^p+
      C\;\lpnorm{\zeta_{k}}{p}^p\eqsp.
      \]
      By Lemma~\ref{BOEM:lem:lpcontrol:dec:rho} and
      (\ref{BOEM:lem:mu:rho:average:rate:tool1}), there exists $C$ s.t. for any $k
      \geq 1$
\begin{equation}
\label{BOEM:eq:lpnormterm1}
I_k^{(1)} \leq C   \left( \tau_{k+1}^{p/2} +  \tau_{k+1}^{p} \mix^{m_{k}/b} \right) \eqsp,
\end{equation}
and since $2/p <1$, convex inequalities yield $\left(\sum_{k=1}^n I_k^{(1)}
\right)^{1/p} \leq C \sqrt{T_{n+1}} + C \sum_{k=1}^n \tau_{k+1}\beta^{m_k
  /pb}$.
By the Minkowski and Jensen inequalities, it holds $I_n^{(2)} \leq \left(
  \sum_{k=1}^n \{I_k^{(1)} \}^{2/p} \right)^{1/2}$.
Hence, by (\ref{BOEM:eq:lpnormterm1}), $I_n^{(2)} \leq C \sqrt{T_{n+1}} + C
\sum_{k=1}^n \tau_{k+1}\beta^{m_k /pb}$. This concludes the proof of
(\ref{BOEM:lem:mu:rho:average:rate:tool2}).
\end{proof}
We write $\Sigma_{n}-s_{\star} = \bar \mu_{n}+\bar\rho_{n}$ with
\begin{equation}
\label{BOEM:eq:def:mu:rho:bar}
\bar \mu_{n} \eqdef \frac{1}{T_{n}}\sum_{k=1}^{n}\tau_{k}\mu_{k-1}\quad \mbox{and}\quad\bar \rho_{n} \eqdef \frac{1}{T_{n}}\sum_{k=1}^{n}\tau_{k}\rho_{k-1}\eqsp.
\end{equation}

\begin{proposition}
\label{BOEM:prop:mu:rho:average:rate}
Assume A\ref{BOEM:assum:strong}, A\ref{BOEM:assum:moment:sup}-($\bar p$),
A\ref{BOEM:assum:obs}-\ref{BOEM:assum:size-block}, A\ref{BOEM:assum:SMCapprox}-($\bar p$) and A\ref{BOEM:assum:stable:fixpoint} for some
$\bar p >2$. For any $p \in (2, \bar p)$,
\[
\sqrt{T_n} \bar{\mu}_n = O_{\rmL_{p}}(1) \eqsp, \qquad \frac{T_n}{n}
\, \bar\rho_{n}\1_{\lim_{n}\param_{n}=\param_{\star}}
=O_{\rmL_{p/2}}(1)O_{\mathrm{a.s}}(1) \eqsp.
\]
\end{proposition}
\begin{proof}
  Set $A \eqdef \left(I-\jacG\right)$. Under A\ref{BOEM:assum:stable:fixpoint},
  $A^{-1}$ exists.  By \eqref{BOEM:eq:def:mu}
  and \eqref{BOEM:eq:def:mu:rho:bar},
\[
A \sqrt{T_{n}}\bar\mu_{n} = -\frac{\tau_{n+1}\mu_{n}}{\sqrt{T_{n}}} +
\frac{1}{\sqrt{T_{n}}}\sum_{k=1}^{n} \tau_{k+1}e_{k} +
\frac{1}{\sqrt{T_{n}}}\sum_{k=1}^{n}\tau_{k}\left(\frac{\tau_{k+1}}{\tau_{k}}-1\right)
\jacG\mu_{k-1}\eqsp.
\]
The result now follows from Proposition~\ref{BOEM:prop:mu:rho:rate},
Lemma~\ref{BOEM:lem:mu:rho:average:rate} and A\ref{BOEM:assum:size-block}.
The proof of the second assertion follows from \eqref{BOEM:eq:def:mu:rho:bar} and
Proposition~\ref{BOEM:prop:mu:rho:rate}.
\end{proof}
Upon noting that $\param_{\star} = \bar\param(s_{\star})$, we may write, for the averaged sequence,
\begin{equation*}
\widetilde{\param}_{n} - \param_{\star} = \Upsilon(\Sigma_{n}-s_{\star}) + \bar\param(\Sigma_{n}) -  \bar\param(s_{\star}) - \Upsilon(\Sigma_{n}-s_{\star})\eqsp.
\end{equation*}
The first term in this decomposition gives
\[
\sqrt{T_{n}}\Upsilon(\Sigma_{n}-s_{\star})\1_{\lim_{n} \param_{n} =\param_{\star}}  =  O_{\rmL_{p}}(1)+ \frac{n}{\sqrt{T_n} }O_{\rmL_{p/2}}(1)O_{\mathrm{a.s}}\left(1\right) \eqsp.
\]
By A\ref{BOEM:assum:stable:fixpoint}\eqref{BOEM:assum:stable:fixpoint:norm}, as for the non averaged sequence, a Taylor expansion with integral remainder term gives the result for the second term. This concludes the proof of Theorem~\ref{BOEM:th:rate:averaged}, Eq.(\ref{BOEM:eq:rate:averaged}).

\section{Acknowledgments}
The authors are grateful to Eric Moulines and Olivier Capp\'e for their
fruitful remarks.

\appendix

\section{Technical results}
\label{BOEM:sec:tech:res}
Proposition~\ref{BOEM:prop:fort:moulines:2003} is exactly \cite[Proposition~$9$]{fort:moulines:2003} applied with a compact set $\paramset$.
\begin{proposition}
\label{BOEM:prop:fort:moulines:2003}
Let $T:\paramset\to\paramset$ and $\lyap$ be a continuous Lyapunov function relatively to $T$ and to $\Staset\subset\paramset$. Assume $\lyap(\Staset)$ has an empty interior and that $\{\param_{n}\}_{n\ge 0}$ is a sequence lying in $\paramset$ such that
\begin{equation}
\label{BOEM:eq:fort:moulines:2003}
\underset{n\to +\infty}{\lim}\left|\lyap(\param_{n+1})-\lyap\circ T(\param_{n})\right| = 0\eqsp.
\end{equation}
Then, there exists $w_{\star}$ such that $\{\param_{n}\}_{n\ge 0}$ converges to $\{\param\in\Staset;\; \lyap(\param)=w_{\star}\}$.
\end{proposition}

The proof of Proposition~\ref{BOEM:prop:expforget} is given in \cite[Proposition~$3.3$]{lecorff:fort:2011-supp}

\begin{proposition}\label{BOEM:prop:expforget}
  Assume A\ref{BOEM:assum:strong}. Let $\chi$, $\widetilde{\chi}$ be two
  distributions on $\left(\Xset,\sigmaX\right)$.  For any measurable
  function $h: \Xset^{2}\times \Yset \to \Rset^d$ and any $\bfy \in
  \Yset^\Zset$ such that $\sup_{x,x'}|h(x,x^{\prime},y_s)| < +\infty$ for any $s
  \in \Zset$
\begin{enumerate}[(i)]
\item \label{BOEM:item:forget-nu-pv} For any $r < s\leq t$ and any $\ell_{1},
  \ell_2\geq 1$,
\begin{equation}\label{BOEM:eq:forget-nu-pv}
\underset{\param\in\paramset}{\sup}\left | \smoothfunc{\widetilde{\chi},r}{\param,s,t}\left(h,\bfy\right) - \smoothfunc{\chi,r-\ell_{1}}{\param,s,t+\ell_{2}}\left(h,\bfy\right) \right | \leq \left (\rho^{s-1-r} + \rho^{t-s}\right )\osc(h_{s})\eqsp.
\end{equation}
\item \label{BOEM:item:forget-limit} For any $\param\in\paramset$, there exists a
  function $\bfy \mapsto \Phi_{\param}(h,\bfy)$ s.t. for any distribution
  $\chi$ on $(\Xset,\sigmaX)$ and any $r<s\leq t$
\begin{equation}\label{BOEM:eq:forget-limit}
\underset{\param\in\paramset}{\sup}\left|\smoothfunc{\chi,r}{\param,s,t}\left(h,\bfy\right) - \smoothfunc{}{\param}\left(h,\shift^s\bfy\right)\right|\leq \left (\rho^{s-1-r} + \rho^{t-s}\right )\osc(h_{s})\eqsp.
\end{equation}
\end{enumerate}
\end{proposition}
\begin{remark}\label{BOEM:rem:lgn}
\begin{enumerate}[(a)]
\item If $\chi=\widetilde{\chi}$, $\ell_{1}=0$ and $\ell_{2}\geq 1$,
  \eqref{BOEM:eq:forget-nu-pv} becomes
\[
\underset{\param\in\paramset}{\sup}\left | \smoothfunc{\chi,r}{\param,s,t}\left(h,\bfy\right) - \smoothfunc{\chi,r}{\param,s,t+\ell_{2}}\left(h,\bfy\right) \right | \leq \rho^{t-s}\osc(h_{s})\eqsp.
\]
\item if $\ell_{2} = 0$ and $\ell_{1}\geq 1$, \eqref{BOEM:eq:forget-nu-pv} becomes
\[
\underset{\param\in\paramset}{\sup}\left | \smoothfunc{\widetilde{\chi},r}{\param,s,t}\left(h,\bfy\right) - \smoothfunc{\chi,r-\ell_{1}}{\param,s,t}\left(h,\bfy\right) \right | \leq \rho^{s-1-r}\osc(h_{s})\eqsp.
\]
\end{enumerate}
\end{remark}

Lemma~\ref{BOEM:lem:bound-phi} is a consequence of \eqref{BOEM:eq:define-Phi}
and of Proposition~\ref{BOEM:prop:expforget}\eqref{BOEM:item:forget-limit}.
\begin{lemma}
\label{BOEM:lem:bound-phi}
Assume A\ref{BOEM:assum:strong}.  Let $r<s\leq t$ be integers, $\param\in \paramset$
and $\bfy\in\Yset^\Zset$, and $h: \Xset^{2}\times \Yset \to \Rset^d$ s.t. for
any $s \in \Zset$, $\sup_{x,x'} |h(x,x^{\prime},y_s)|< \infty$. Then
\[
\left|\smoothfunc{\chi,r}{\param,s,t}\left(h,\bfy\right)\right|\leq
\underset{(x,x^{\prime})\in\Xset^2}{\sup}\left|h(x,x^\prime,y_s)\right|
\eqsp,\,
\left|\smoothfunc{}{\param}\left(h,\shift^s\bfy\right)\right|\leq\underset{(x,x^{\prime})\in\Xset^2}{\sup}\left|h(x,x^\prime,y_s)\right|\eqsp.
\]
\end{lemma}

For any $L\geq 1$, $m\geq 1$ and any
distribution $\chi$ on $(\Xset,\sigmaX)$, define
\begin{equation}
\label{BOEM:eq:def:kappa}
\kappa_{L,m}^{\chi}(\boldsymbol{\param},\bfY) \eqdef \smoothfunc{\chi,L-m}{\boldsymbol{\param},L,L+m}(S,\bfY) -\CE[]{}{}{\smoothfunc{\chi,-m}{\upsilon,0,m}(S,\bfY)}_{\upsilon=\boldsymbol{\param}}\eqsp.
\end{equation}
 We introduce the $\sigma$-algebra $\widetilde{\mathcal{F}}_{T_{n}}$ defined by
 \begin{equation}
 \label{BOEM:eq:Zfield:tilde}
 \widetilde{\mathcal{F}}_{T_{n}}\eqdef\sigma\{\mathcal{F}_{T_{n}}^{\bfY},\mathcal{H}_{T_{n}}\}\eqsp,
 \end{equation}
 where $\mathcal{F}_{T_{n}}$ is given by \eqref{BOEM:eq:Zfield} and where $\mathcal{H}_{T_{n}}$ is independent from $\bfY$ (the $\sigma$-algebra $\mathcal{H}_{T_{n}}$ is generated by the random variables independent from the
observations $\bfY$ used to produce the Monte Carlo approximation of $\{S_{k-1}\}_{k=1}^{n}$).
Hence, for any positive integer $m$ and any $B\in\mathcal{G}_{T_{n}+m}^{\bfY}$,
since $\mathcal{H}_{T_{n}}$ is independent from $B$ and from
$\mathcal{F}_{T_{n}}^{\bfY}$,  $\PPim(B\vert \widetilde{\mathcal{F}}_{T_{n}})= \PPim(B\vert\mathcal{F}_{T_{n}}^{\bfY})$.
Hence, the mixing coefficients defined in \eqref{BOEM:eq:def:mixing} are such that
\[
\mix{}(\mathcal{G}_{T_n+m}^{\bfY},\widetilde{\mathcal{F}}_{T_n}) = \mix(\mathcal{G}_{T_n+m}^{\bfY},\mathcal{F}_{T_n^{\bfY}})\eqsp.
\]
Note that $\param_{n}$ is $\widetilde{\mathcal{F}}_{T_{n}}$- measurable and that $\widetilde{S}_{n}$ is $\widetilde{\mathcal{F}}_{T_{n+1}}$-measurable.
\begin{lemma}
\label{BOEM:lem:lpcontrol}
Assume A\ref{BOEM:assum:strong}, A\ref{BOEM:assum:moment:sup}-($\bar p$) and
A\ref{BOEM:assum:obs} for some
$\bar p >2$. Let $p\in(2,\bar p)$. There exists a constant $C$ s.t. for any
distribution $\chi$ on $(\Xset,\sigmaX)$, any $m\geq 1$,
$k, \ell\geq 0$ and any $\paramset$-valued $\widetilde{\mathcal{F}}_{0}^{\bfY}$-measurable
r.v. $\boldsymbol{\param}$,
\begin{equation*}
\lpnorm{\sum_{u=1}^{k}  \kappa_{2um+\ell,m}^{\chi}(\boldsymbol{\param},\bfY)}{p}\leq C\left[\sqrt{\frac{k}{m}}+k\mix^{m \, \Delta p}\right]\eqsp,
\end{equation*}
where $\Delta p \eqdef \frac{\bar p - p}{p\bar p}$ and $\mix$ is given by A\ref{BOEM:assum:obs}.
\end{lemma}

\begin{proof}
  For ease of notation $\chi$ is dropped from the notation
  $\kappa_{2um,m}^{\chi}$. By the Berbee Lemma (see \cite[Chapter
  $5$]{rio:1999}), for any $m\geq 1$, there exists a $\paramset$-valued r.v.
  $\boldsymbol{\upsilon}^{\star}$ on $(\Omega,\mathcal{F},\PPim)$ independent
  from $\mathcal{G}_{m}^{\bfY}$ (see \eqref{BOEM:eq:Zfield}) s.t.
\begin{equation}
\label{BOEM:eq:berb}
\PP_{}\left\{\boldsymbol{\param}\neq\boldsymbol{\upsilon}^{\star}\right\} =\underset{B\in\mathcal{G}_{m}^{\bfY}}{\sup}\,|\PPim(B| \sigma(\boldsymbol{\param})) - \PPim(B)|\eqsp.
\end{equation}
Set $L_{u} \eqdef 2um+\ell$. We write
\begin{multline}
  \sum_{u=1}^{k} \kappa_{L_{u},m}(\boldsymbol{\param},\bfY)= \sum_{u=1}^{k}
  \left\{\smoothfunc{\chi,L_{u}-m}{\boldsymbol{\param},L_{u},L_{u}+m}(S,\bfY) -
    \smoothfunc{\chi,L_{u}-m}{\boldsymbol{\upsilon}^{\star},L_{u},L_{u}+m}(S,\bfY)\right\}
  \\+ \sum_{u=1}^{k} \kappa_{L_{u},m}(\boldsymbol{\upsilon}^{\star},\bfY) +
  k\left\{\CE[]{}{}{\smoothfunc{\chi,-m}{\upsilon,0,m}(S,\bfY)}_{\upsilon
      =
      \boldsymbol{\upsilon}^{\star}}-\CE[]{}{}{\smoothfunc{\chi,-m}{\upsilon,0,m}(S,\bfY)}_{\upsilon
      = \boldsymbol{\param}}\right\}\label{BOEM:eq:sumkappa}\eqsp.
\end{multline}
By the Holder's inequality with $a\eqdef \bar{p}/p$ and $b^{-1}\eqdef 1 - a^{-1}$,
\begin{multline*}
  \lpnorm{\smoothfunc{\chi,L-m}{\boldsymbol{\param},L,L+m}(S,\bfY) -
    \smoothfunc{\chi,L-m}{\boldsymbol{\upsilon}^{\star},L,L+m}(S,\bfY)}{p} \\
  \leq \lpnorm{\smoothfunc{\chi,L-m}{\boldsymbol{\param},L,L+m}(S,\shift^{T}
     \bfY) -
    \smoothfunc{\chi,L-m}{\boldsymbol{\upsilon}^{\star},L,L+m}(S,\bfY)}{\bar
    p}
  \PP_{}\left\{\boldsymbol{\param}\neq\boldsymbol{\upsilon}^{\star}\right\}^{\Delta
    p}\eqsp.
\end{multline*}
By A\ref{BOEM:assum:moment:sup}-($\bar p$), A\ref{BOEM:assum:obs}, \eqref{BOEM:eq:def:mixing} and \eqref{BOEM:eq:berb}, there exists a constant $C_{1}$ s.t. for any $m,L\geq
1$, any distribution $\chi$ and any $\paramset$-valued
$\widetilde{\mathcal{F}}_{0}^{\bfY}$-measurable r.v.  $\boldsymbol{\param}$,
\begin{equation*}
\lpnorm{\smoothfunc{\chi,L-m}{\boldsymbol{\param},L,L+m}(S,\bfY) - \smoothfunc{\chi,L-m}{\boldsymbol{\upsilon}^{\star},L,L+m}(S,\bfY)}{\bar p}
    \leq C_{1} \mix^{m \Delta p}\eqsp.
    \end{equation*}
    Similarly, there exists a constant $C_{2}$ s.t. for any $m\geq 1$, any
    distribution $\chi$ and any $\paramset$-valued
    $\widetilde{\mathcal{F}}_{0}^{\bfY}$-measurable r.v.  $\boldsymbol{\param}$,
\begin{equation*}
\lpnorm{\CE[]{}{}{\smoothfunc{\chi,-m}{\upsilon,0,m}(S,\bfY)}_{\upsilon =
      \boldsymbol{\upsilon}^{\star}}-\CE[]{}{}{\smoothfunc{\chi,-m}{\upsilon,0,m}(S,\bfY)}_{\upsilon =
      \boldsymbol{\param}}}{p} \leq C_{2}  \mix^{m \Delta p}\eqsp.
      \end{equation*}
      Let us consider the second term in \eqref{BOEM:eq:sumkappa}. For any $u\geq 1$
      and any $\upsilon\in\paramset$, the r.v.
      $\kappa_{L_{u},m}(\upsilon,\bfY)$ is a measurable function of $\bfY_{i}$
      for all $L_{u}-m+1\leq i\leq L_{u}+m$. Since $L_{u}\geq 2um$, for any
      $\upsilon\in\paramset$, $\sum_{u=1}^{k} \kappa_{L_{u},m}(\upsilon,\bfY)$
      is $\mathcal{G}_{m}^{\bfY}$-measurable. $\boldsymbol{\upsilon}^{\star}$ is
      independent from $\mathcal{G}_{m}^{\bfY}$ so that:
   \begin{equation*}
\lpnorm{\sum_{u=1}^{k} \kappa_{L_{u},m}(\boldsymbol{\upsilon}^{\star},\bfY)}{p}=\CE[]{}{}{\CE[]{}{}{\left|\sum_{u=1}^{k}  \kappa_{L_{u},m}(\upsilon,\bfY)\right|^{p}}_{\upsilon=\boldsymbol{\upsilon}^{\star}}}^{1/p}\eqsp.
\end{equation*}
Define the strong mixing coefficient (see \cite{davidson:1994})
\[
\mix[a]^{\bfY}(r)
\eqdef\underset{u\in\Zset}{\sup}\,\underset{(A,B)\in\mathcal{F}_{u}^{\bfY}\times\mathcal{G}_{u+r}^{\bfY}}{\sup}\,|\PPim(A\cap
B) - \PPim(A)\PPim(B)| \eqsp, r \geq 0 \eqsp.
\] Then,
\cite[Theorem $14.1$, p.210]{davidson:1994} implies that for any $m \geq 1$,
the strong mixing coefficients of the sequence ${\bf
  \kappa_{(m)}}\eqdef\{\kappa_{L_{u},m}(\upsilon,\bfY)\}_{u\geq 1}$ satisfies
$\mix[a]^{{\bf \kappa_{(m)}}}(i) \leq \mix[a]^{\bfY}(2(i-1)m)$. Furthermore, by
\cite[Theorem 2.5]{rio:1999},
\[
\lpnorm{\sum_{u=1}^{k} \kappa_{L_{u},m}(\upsilon,\bfY)}{p}\leq
(2kp)^{1/2} \left(\int_{0}^{1}\left[N_{(m)}(t)\wedge k\right]^{p/2}
  \mathcal{Q}_{\upsilon,m}^{p}(t)\rmd t \right)^{1/p}\eqsp,
\]
where $N_{(m)}(t)\eqdef \sum_{i\geq 1} \1_{\alpha^{{\bf \kappa_{(m)}}}(i)>t}$
and $\mathcal{Q}_{\upsilon,m}$ denotes the inverse of the tail function
$t\mapsto \PPim(|\kappa_{L_{u},m}(\upsilon,\bfY)|\geq t)$. The sequence $\bfY$
being stationary, this inverse function does not depend on $u$. By
A\ref{BOEM:assum:obs} and the inequality $\mix[a]^{\bfY}(r)\leq
\mix^{\bfY}(r)$ (see e.g. \cite[Chapter $13$]{davidson:1994}), there exist
$\mix\in[0,1)$ and $C\in(0,1)$ s.t. for any $u,m\geq 1$,
\[
N_{(m)}(u)\leq \sum_{i\geq 1} \1_{\mix[a]^{\bfY}(2(i-1)m)>u} \leq\sum_{i\geq 1}\1_{C\mix^{2(i-1)m}>u}\leq\left(\frac{\log u-\log C}{2m\log \mix}\right)\vee 0\eqsp.
\]
Let $U$ be a uniform r.v. on $[0,1]$. Observe that $C\mix^{2mk}<1$.  Then, by
the Holder inequality applied with $a\eqdef \bar{p}/p$ and $b^{-1}\eqdef 1 -
a^{-1}$,
\begin{align*}
  & \left\| \left[N_{(m)}(U)\wedge k\right]^{1/2} \mathcal{Q}_{\upsilon,m}(U) \right\|_p \eqdef  \left(\int_{0}^{1}\left[N_{(m)}(u)\wedge k\right]^{p/2} \mathcal{Q}_{\upsilon,m}^{p}(u)\rmd u \right)^{1/p}  \\
  &\leq  \left[\frac{-1}{2m\log \mix}\right]^{1/2 }\left\|\mathcal{Q}_{\upsilon,m}(U)\left(-\log \frac{U}{C}\right)^{1/2}\1_{(C\beta^{Cmk},C)}(U)\right\|_p\\
  &\hspace{7cm}+ k^{1/2} \left\| \mathcal{Q}_{\upsilon,m}(U) \1_{U\leq C\mix^{2mk}}  \right\|_p\eqsp,\\
  &\leq \left\{(C\mix^{2mk})^{\Delta p}k^{1/2} + \left[\frac{-1}{2m\log
        \mix}\right]^{1/2} \left\| \left(-\log \frac{U}{C}\right)^{1/2} \1_{(C\beta^{Cmk},C)}(U)
    \right\|_{p b}\right\}\\
    &\hspace{7cm}\times\left\| \mathcal{Q}_{\upsilon,m}(U)\right\|_{\bar p}\eqsp.
\end{align*}
Since $U$ is uniform on $[0,1]$, $\mathcal{Q}_{\upsilon,m}(U)$ and $|\kappa_{L_{u},m}(\upsilon,\bfY)|$ have
the same distribution, see \cite{rio:1999}. Then, by Lemma~\ref{BOEM:lem:bound-phi} and
A\ref{BOEM:assum:moment:sup}-($\bar p$), there exists a constant $C$ s.t. for any
$\upsilon\in\paramset$, any $m\geq1$,
\[
\underset{\upsilon\in\paramset}{\sup}\;  \left\|\mathcal{Q}_{\upsilon,m}(U)
\right\|_{\bar p}
\leq C\,\lpnorm{\sup_{x,x' \in \Xset^2} \,
  |S(x,x',\bfY_{0})}{\bar p}\eqsp,
\]
which concludes the proof.
\end{proof}

\begin{lemma}
\label{BOEM:lem:lpcontrol:dec:rho}
Assume A\ref{BOEM:assum:strong},  A\ref{BOEM:assum:moment:sup}-($\bar p$) and A\ref{BOEM:assum:obs} for some $\bar p>2$. Let $p\in(2,\bar p)$. There exists a constant $C$ s.t.  for any $n\geq 1$, any $1\leq m_{n}\leq \tau_{n+1}$ and any distribution $\chi$ on $(\Xset,\sigmaX)$,
\begin{equation*}
\lpnorm{\frac{1}{\tau_{n+1}}\sum_{t=2m_n}^{2v_{n}m_{n}} \kappa_{t,m_{n}}^{\chi}(\param_{n},\shift^{T_{n}}  \bfY)}{p}\leq C\left[\frac{1}{\sqrt{\tau_{n+1}}} + \mix^{m_{n}\Delta p}\right]\eqsp,
\end{equation*}
where $\kappa_{L,m}^{\chi}$ and $\mix$ are defined by \eqref{BOEM:eq:def:kappa} and
A\ref{BOEM:assum:obs}, $v_{n}\eqdef
\left\lfloor\frac{\tau_{n+1}}{2m_{n}}\right\rfloor$ and $\Delta p \eqdef
\frac{\bar p - p}{p\bar p}$.
\end{lemma}
\begin{proof}
We write,
\[
\lpnorm{\sum_{t=2m_n}^{2v_{n}m_{n}}
  \kappa_{t,m_{n}}^{\chi}(\param_{n},\shift^{T_{n}} \bfY)}{p} \leq
\sum_{\ell=0}^{2m_{n}-1}\lpnorm{ \sum_{u=1}^{v_{n}-1}
  \kappa_{2um_{n}+\ell,m_{n}}^{\chi}(\param_{n},\shift^{T_{n}} \bfY)}{p}\eqsp.
\]
Observe that by definition $\param_{n}$ is $\widetilde{\mathcal{F}}_{T_{n}}^{\bfY}$-measurable. Then, by
Lemma~\ref{BOEM:lem:lpcontrol}, there exists a constant $C$ s.t. for any $m_{n}\geq
1$ and any $\ell\geq 0$,
\[
\lpnorm{ \sum_{u=1}^{v_{n}-1}
  \kappa_{2um_{n}+\ell,m_{n}}^{\chi}(\param_{n},\shift^{T_{n}} \bfY)}{p}\leq
C\left[\sqrt{\frac{v_{n}}{m_{n}}} + v_{n}\beta^{m_{n}\Delta p}\right]\eqsp.
\]
The proof is concluded upon noting that  $\tau_{n+1}\geq 2m_{n}v_{n}$.
\end{proof}

\begin{lemma}
\label{BOEM:lem:lpcontrol:dec}
Assume A\ref{BOEM:assum:strong}, A\ref{BOEM:assum:moment:sup}-($\bar p$) and
A\ref{BOEM:assum:obs} for some $\bar p >2$. For any $p \in (2,
\bar p]$, there exists a constant $C$ s.t. for any $n\geq 1$, any $1\leq
m_{n}\leq q_{n}\leq \tau_{n+1}$ and any distribution $\chi$ on
$(\Xset,\sigmaX)$,
\begin{equation*}
\lpnorm{\bar S_{\tau_{n+1}}^{\chi,T_{n}}(\param_{n}, \bfY)-\mapS(\param_n)-\widetilde{\rho}_{n}}{p}
\leq C\left[ \rho^{m_{n}} + \frac{m_{n}}{\tau_{n+1}}+\frac{\tau_{n+1}-q_{n}}{\tau_{n+1}}\right]\eqsp,
\end{equation*}
where $\widetilde{\rho}_{n} \eqdef \tau_{n+1}^{-1}\sum_{t=2m_n}^{q_{n}} \kappa_{t,m_{n}}^{\chi}(\param_{n},\shift^{T_{n}}\bfY)$
and $\kappa_{L,m}^{\chi}$ is defined by \eqref{BOEM:eq:def:kappa}.
\end{lemma}

\begin{proof}
  By  \eqref{BOEM:eq:rewrite:barS} and \eqref{BOEM:eq:define-Phi}, $ \bar
  S_{\tau_{n+1}}^{\chi,T_{n}}(\param_{n}, \bfY)-\mapS(\param_n)
  -\widetilde{\rho}_{n} = \sum_{i=1}^4 g_{i,n}$
where
\begin{align*}
  g_{1,n} & \eqdef
  \frac{1}{\tau_{n+1}}\sum_{t=1}^{\tau_{n+1}}\left(\smoothfunc{\chi,0}{\param_{n},t,\tau_{n+1}}(S,\shift^{T_n}
   \bfY)
  -\smoothfunc{\chi,t-m_n}{\param_{n},t,t+m_n}(S,\shift^{T_n}
   \bfY)\right)   \eqsp, \\
  g_{2,n} &\eqdef \frac{1}{\tau_{n+1}}\sum_{t=1}^{2m_n-1} \left(
    \smoothfunc{\chi,t-m_n}{\param_{n},t,t+m_n}(S,\shift^{T_n}  \bfY) -
    \CE[]{}{}{\smoothfunc{\chi,-m_{n}}{\param,0,m_{n}}(S,\bfY)}_{\param =
      \param_n} \right)\eqsp,\\
  g_{3,n} &\eqdef
  \frac{1}{\tau_{n+1}}\sum_{t=q_{n}+1}^{\tau_{n+1}}\left(\smoothfunc{\chi,t-m_n}{\param_{n},t,t+m_n}(S,\shift^{T_n}
     \bfY) -
    \CE[]{}{}{\smoothfunc{\chi,-m_{n}}{\param,0,m_{n}}(S,\bfY)}_{\param =
      \param_n}\right)\eqsp, \\
  g_{4,n} & \eqdef
  \CE[]{}{}{\smoothfunc{\chi,-m_{n}}{\param,0,m_{n}}(S,\bfY)}_{\param =
    \param_n} - \mapS(\param_n) \eqsp.
\end{align*}
In the case $\tau_{n+1} > 2 m_n$, it holds
\begin{align*}
  & \tau_{n+1} \left| g_{1,n} \right| \leq
  \sum_{t=\tau_{n+1}-m_n+1}^{\tau_{n+1}} \left( \rho^{m_n-1} +
    \rho^{\tau_{n+1}-t}\right) \osc(S_{t+T_n}) \\
  &\hspace{2cm} + \sum_{t=1}^{m_n} \left(\rho^{m_n} + \rho^{t-1} \right)
  \osc(S_{t+T_n}) +\;2 \rho^{m_n-1}
  \sum_{t=m_n+1}^{\tau_{n+1}-m_n}
  \osc(S_{t+T_n})\eqsp,
\end{align*}
where we used Proposition~\ref{BOEM:prop:expforget}\eqref{BOEM:item:forget-nu-pv} and
Remark~\ref{BOEM:rem:lgn}  in the last inequality. By A\ref{BOEM:assum:moment:sup}-($\bar
p$) and A\ref{BOEM:assum:obs}, there exists $C$ s.t.
$\lpnorm{g_{1,n}}{p} \leq C \left( \rho^{m_n} + \tau_{n+1}^{-1}\right)$. The same bound hold in the case $\tau_{n+1} \leq 2 m_n$.  For $g_{2,n}$ and $g_{3,n}$, we use the bounds
\begin{align*}
  \left|\smoothfunc{\chi,t-m_n}{\param_{n},t,t+m_n}(S,\shift^{T_n}  \bfY)
    - \CE[]{}{}{\smoothfunc{\chi,-m_{n}}{\param,0,m_{n}}(S,\bfY)}_{\param
      =
      \param_n} \right|&\\
  &\hspace{-6cm}\leq \underset{(x,x')\in\Xset^{2}}{\sup}\left|S(x,x',Y_{T_{n}+t})\right|+
  \CE[]{}{}{\underset{(x,x')\in\Xset^{2}}{\sup}\left|S(x,x',Y_{0})\right|}\eqsp.
\end{align*}Then, by
A\ref{BOEM:assum:obs},

\begin{multline*}
  \lpnorm{\smoothfunc{\chi,t-m_n}{\param_{n},t,t+m_n}(S,\shift^{T_n}
    \bfY) -
    \CE[]{}{}{\smoothfunc{\chi,-m_{n}}{\param,0,m_{n}}(S,\bfY)}_{\param =
      \param_n}}{p}\\
  \leq 2 \lpnorm{\underset{(x,x')\in\Xset^{2}}{\sup}\left|S(x,x',Y_{0})
    \right|}{p}\eqsp,
\end{multline*}
and the RHS is finite under A\ref{BOEM:assum:moment:sup}-($\bar p$). Finally,
\begin{equation*}
  \left|g_{4,n} \right| \leq 2 \rho^{m_n-1}\CE[]{}{}{\osc(S_{0})} \eqsp,
\end{equation*}
where we used Theorem~\ref{BOEM:th:LGN}. This concludes the
proof.
\end{proof}

\bibliographystyle{plain}
\bibliography{./onlineblock}

\end{document}